\documentclass[11pt,oneside]{amsart}
\usepackage{wasysym}
\usepackage[top=20mm,bottom=15mm,inner=20mm,outer=40mm]{geometry}
\usepackage{graphicx}
\usepackage{marginnote}
\usepackage{hyperref}
\usepackage{mathrsfs}
\usepackage{cases}
\usepackage{amsmath, amsfonts, amssymb, amsthm ,mathrsfs, bbm,dsfont,units}
\usepackage{color}

\usepackage{empheq}

\newtheorem{theorem}{Theorem}[section]
\newtheorem{corollary}[theorem]{Corollary}
\newtheorem{lemma}[theorem]{Lemma}
\newtheorem{proposition}[theorem]{Proposition}

\theoremstyle{definition}

\newtheorem{example}[theorem]{Example}

\theoremstyle{remark}
\newtheorem{remark}[theorem]{Remark}

\usepackage{delarray}
\usepackage{enumerate}

\usepackage{color}
\definecolor{dp}{RGB}{34,139,34}
\definecolor{kb}{RGB}{0,155,255}
\definecolor{tj}{RGB}{255,0,255}
\definecolor{pk}{rgb}{0,0,0.6}

\numberwithin{equation}{section}
\numberwithin{theorem}{section}
\numberwithin{figure}{section}

\newcommand{\ud}{\,{\rm d}}
\newcommand{\rf}[1]{\eqref{#1}}

\newcommand{\bbfR}{{\mathbb R}}

\newcommand{\bbfN}{{\mathbb N}}
\newcommand{\vf}{{\varphi}}
\newcommand{\ve}{{\varepsilon }}
\newcommand{\se}{{\text{\rm{e}}}}

\newcommand{\Aa}{{\mathcal A}}

\newcommand{\R}{{\mathbb R}}

\date{\rule{0pt}{15pt}\today}
\newcommand{\Rd}{{{\mathbb R}^d}}
\newcommand{\Rdz}{{{\mathbb R}^d_0}}
\newcommand{\RR}{\mathbb{R}}
\newcommand{\NN}{\mathbb{N}}

\newcommand{\al}{{\alpha}}
\newcommand{\indyk}{\mathbf{1}}

\def \eps{\varepsilon}

\def \tp{\tilde{p}}
\def \tP{\tilde{P}}

\title[Self-similar solution]{Self-similar solution for 
Hardy 
operator}
\author[K. Bogdan]{Krzysztof Bogdan}
\address{Faculty of Pure and Applied Mathematics,
Wroc\l aw University of Science and Technology,
Wyb. Wyspia\'nskiego 27, 50-370 Wroc\l aw, Poland}
\email{krzysztof.bogdan@pwr.edu.pl}
\author[T. Jakubowski]{Tomasz Jakubowski}
\address{Faculty of Pure and Applied Mathematics,
Wroc\l aw University of Science and Technology,
Wyb. Wyspia\'nskiego 27, 50-370 Wroc\l aw, Poland}
\email{tomasz.jakubowski@pwr.edu.pl}
\author[P. Kim]{Panki Kim}
\address{Department of Mathematical Sciences and Research Institute of Mathematics,
Seoul National University, Seoul 08826, Republic of Korea}
\email{pkim@snu.ac.kr}
\author[D. Pilarczyk]{Dominika Pilarczyk}
\address{Faculty of Pure and Applied Mathematics,
Wroc\l aw University of Science and Technology,
Wyb. Wyspia\'nskiego 27, 50-370 Wroc\l aw, Poland}
\email{dominika.pilarczyk@pwr.edu.pl}
\thanks{The first author was supported by grant 2017/27/B/ST1/01339 of National Science Centre, Poland.
The second author was supported by grant 2015/18/E/ST1/00239 of National Science Centre, Poland.
The third author was supported by the National Research Foundation of Korea (NRF), grant funded by the Korea government (MSIP) No. NRF-2021R1A4A1027378.}


\begin{document}

\begin{abstract}
We describe the large-time asymptotics of solutions to the heat equation  for the fractional Laplacian with added subcritical or even critical Hardy-type potential. The asymptotics is governed by a self-similar solution of the equation, obtained as a normalized limit at the origin of the kernel of the corresponding Feynman-Kac semigroup.
\end{abstract}

\maketitle

\section{Introduction}\label{s.i}

\subsection{Main results and structure of the paper}\label{s:sg}
Let $d\in \bbfN:=\{1,2,\ldots\}$, $\alpha\in (0,2)$ and $\alpha<d$.
We 
consider the semigroup $\tilde P_t$, $t>0$, of the following Hardy operator 
on $\Rd$,
   \begin{equation}
   \label{eq:SchrOp}
   \Delta^{\alpha/2}+\kappa|x|^{-\alpha}.
   \end{equation} 
We call $\kappa$, and $\Delta^{\alpha/2}+\kappa|x|^{-\alpha}$, subcritical  if $\kappa<\kappa^*$, critical  if $\kappa=\kappa^*$ and supercritical  if $\kappa>\kappa^*$.
Here $\Delta^{\alpha/2}:=-(-\Delta)^{\alpha/2}$ is the fractional Laplacian,
\begin{equation*}
   \kappa^*:=\frac{2^{\alpha} \Gamma((d+\alpha)/4)^2 }{\Gamma((d-\alpha)/4)^{2}},
\end{equation*}
and $\Gamma(t)=\int_0^\infty y^{t-1}e^{-y}\ud y$ is the Gamma function.
It is well known that $\kappa^*$ is the best constant in the Hardy inequality for the quadratic form of $\Delta^{\alpha/2}$,
see Herbst \cite[Theorem~2.5]{MR0436854}, Beckner \cite[Theorem~2]{MR1254832} or Yafaev \cite[(1.1)]{MR1717839}; see also Frank and Seiringer \cite[Theorem~1.1]{MR2469027} and Bogdan, Dyda and Kim \cite[Proposition~5]{MR3460023}. 
Following \cite[Section~4]{MR3460023}, for $\beta\in [0,d-\alpha]$ we let
\begin{equation}\label{e.dkd}
\kappa_\beta=\frac{2^\alpha\Gamma((\beta+\alpha)/2) \Gamma((d-\beta)/2)}{\Gamma(\beta/2)\Gamma((d-\beta-\alpha)/2)},
\end{equation}
where $\kappa_0 =\kappa_{d-\alpha}= 0$, according to the convention $1/\Gamma(0)=0$.  The function $\beta \mapsto \kappa_\beta$ is increasing on $[0,(d-\alpha)/2]$, decreasing on $[(d-\alpha)/2,d-\alpha]$, and $\kappa_\beta = \kappa_{d-\alpha-\beta}$, see  \cite[Proof of Proposition~5]{MR3460023}.  The maximal or {\it critical} value of $\kappa_\beta$ is, therefore, $\kappa_{(d-\alpha)/2}=\kappa^*$, and
for each $\kappa\in [0,\kappa^*]$ there is  a unique number $\delta$ such that 
\begin{equation}\label{e.cdk}
0\le\delta\leq (d-\alpha)/2 \quad \text{ and } \quad
\kappa=\kappa_\delta=\frac{2^\alpha\Gamma((\delta+\alpha)/2) \Gamma((d-\delta)/2)}{\Gamma(\delta/2)\Gamma((d-\delta-\alpha)/2)}.
\end{equation}
In what follows, $\delta$ and $\kappa$
shall satisfy \eqref{e.cdk}. We let
$h(x)=h_\delta(x):=|x|^{-\delta}$, $x\in \Rd$. 
By Bogdan, Grzywny, Jakubowski and Pilarczyk \cite{MR3933622},
the Schr\"odinger operator \eqref{eq:SchrOp} has heat kernel $\tilde{p}$ with singularity at the origin in $\Rd$ and sharp explicit estimates (given by
\eqref{eq:mainThmEstz} below).
The first main result 
of the present paper 
is a description of the limiting behavior of $\tilde{p}$, as follows.
\begin{theorem}\label{t:eta2}
The limit $\Psi_t(x):=
\lim\limits_{y \to 0} \frac{\tilde{p}(t,x,y)}{h(y)}$ exists whenever $0\le\delta\leq \frac{d-\alpha}{2}$, $t>0$, $x\in \Rd$. 
\end{theorem}
The proof of Theorem~\ref{t:eta2} is given in Section~\ref{s:csss}.
The function $\Psi_t(x)$ is a self-similar semigroup solution of the heat equation for the Hardy operator,
as we assert in \rf{e.s} and \rf{ss form eta1} below.
It has an important application to large-time asymptotics of the semigroup $\tilde P_t$, which we now present. To this end we consider Doob--conditioned and weighted 
$L^q$ spaces. Let
$$H=\max\{1,h\}.$$
As usual, $L^1=L^1(\Rd,\ud x)$, $L^1(H)=L^1(\Rd,H(x)\ud x)$, etc. We have $L^1(H)=\{f/H: f\in L^1\}=L^1(h) \cap L^1$.
We then define, for $1\le q <\infty$, 
\begin{equation}\label{normqh}
	\|f\|_{q,h }:=\|f/h\|_{L^q(h^2)}
=\bigg( \int_{\Rd} |f(x)|^q h^{2-q}(x)\ud x \bigg) ^{\frac{1}{q}}
=\|f\|_{L^q(h^{2-q})} ,
\end{equation}      
and, for $q=\infty$,
\begin{equation*}
	\|f\|_{\infty, h}:={\rm ess} \sup_{x\in \bbfR^d} |f(x)|/h(x).
	\end{equation*}
Of course, $\|f\|_{2,h}=\|f\|_2$ and $\|f\|_{1,h}=\|f\|_{L^1(h)}$. 
For $f\in L^1(H)$ we let
\begin{equation}\label{e.dtp}
\tilde P_t f(x):=\int_\Rd \tilde p(t,x,y)f(y) \ud y,\quad t>0,\; x\in \Rd\setminus \{0\}.
\end{equation}
Our second main result is the following large-time asymptotics for $\tilde P_t$.
\begin{theorem}\label{lim norm th}
If $f\in L^1(H)$, $A=\int_{\Rd}f (x)  h(x)   \ud x$, $u(t,x)=\tilde P_t f(x)$,
 and $q\in[1,\infty)$, then	
	\begin{equation}\label{e.as}
		\lim_{t\rightarrow \infty }t^{\frac{d-2\delta}{\alpha}(1-\frac{1}{q})}\| u(t, \cdot)-A\Psi_t\|_{q, h}=0.
	\end{equation}
\end{theorem}

The structure of the paper is as follows. The proof of  Theorem~\ref{lim norm th} is given at the end of Section~\ref{s.ssa}, where we also show that the result is optimal. In Section~\ref{s:csss} we state and prove Theorem \ref{t:eta}, of which Theorem \ref{t:eta2} is a direct consequence.
In Section~\ref{s.fatp} we discuss the Feynman-Kac semigroup $\tilde{P}_t$ from the point of view of functional analysis, in particular we prove hypercontractivity of the semigroup in Theorem~\ref{hyperc}, which is then used in Section~\ref{s.ssa}. The last main result of the paper,
Theorem~\ref{thm:intu} in Section~\ref{s.psss}, gives an explicit formula for the potential $\int_0^\infty  \Psi_t(x) \ud t$ of the self-similar solution.
Notably, Theorem~\ref{thm:intu} and Corollary~\ref{cor:intPsi} further the integral analysis which is the foundation of \cite{MR3933622}. They were inspired by one of our earlier attempts to prove Theorem~\ref{t:eta2} and are particularly interesting for $\kappa=\kappa^*$, see \eqref{e.pcc}.

\subsection{Motivation and methods}
\label{s:Mm}
The classical result of Baras and Goldstein \cite{MR742415} asserts the
existence of nontrivial nonnegative solutions of the heat equation $
\partial_t
 = \Delta + \kappa|x|^{-2}$ in $\Rd$ for (subcritical) $\kappa \in [0, (d - 2)^2/4]$, and non-existence of such solutions for (supercritical) $\kappa>(d-2)^2/4$. 
Later on, the upper and lower bounds for the heat kernel of the subcritical Hardy operator $\Delta + \kappa|x|^{-2}$ were obtained by Liskevich and Sobol \cite{MR1953267}, Milman and Semenov \cite{MR2064932,MR2180080}, Moschini and Tesei \cite{MR2328115}, Filippas, Moschini and Tertikas \cite{MR2308757}.

The classical Hardy operator $\Delta+\kappa|x|^{-2}$ plays a distinctive role in limiting and self-similar phenomena in
probability \cite{2017-LM-EP} and partial differential equations \cite{MR3020137}. This is related to the scaling of the corresponding heat kernel, which is the same as for the Gauss-Weierstrass kernel, and to the asymptotics at the origin in $\Rd$, which is very different. Such applications motivate our work on the Hardy perturbation of the fractional Laplacian. In fact the paper \cite{MR3933622} was a preparation for the present work, which now comes to fruition.

The strategy of the proof of Theorem~\ref{t:eta2} is to 
prove and use the existence of a stationary density of a corresponding Ornstein-Uhlenbeck semigroup. Then the large-time asymptotics of the Ornstein-Uhlenbeck semigroup yields the asymptotics of $\tilde p$ at the origin. To the best of our knowledge the approach is new
 and should apply to other heat kernels with scaling. 

Let us also comment on Theorem~\ref{lim norm th}. We start by recalling the initial value problem for the classical heat equation,
\begin{equation}\label{c.h.eq}
	\begin{cases}
		\partial_t u(x,t)=\Delta u(x,t) , \quad x\in \bbfR^d, \quad t>0 ,\\
		u(x,0)=f(x) .
	\end{cases}
\end{equation}  
For 
$f \in L^1 (\bbfR^d )$ the following is asymptotics is well-known:
\begin{equation}\label{heat lim}
	\lim_{t\rightarrow \infty }t^{\frac{d}{2}(1-\frac{1}{p})}\| u(t,\cdot)-Mg_t\|_{L^p(\bbfR^d)} =0,
\end{equation}
see, e.g., Giga, Giga and Saal \cite[Theorem in Sect. 1.1.4]{Giga} or Duoandikoetxea and Zuazua \cite{DZ}). Here  $p\in [1, \infty ]$, $M:=\int_{\bbfR^d}f(x)  \ud x $, $u(t,x) =  g_t \ast f(x)$
is the semigroup solution of \rf{c.h.eq}, and $g_t(x)=\big(4\pi t\big)^{-d/2}\exp( -|x|^2/(4 t))$ is the Gauss-Weierstrass kernel. Of course, we have \textit{scaling}: $g_t(x)=t^{-d/2}g_1(t^{-1/2}x)$, that is, the function is \textit{self-similar}. The function also satisfies the first equation in \eqref{c.h.eq}. We can consider \eqref{heat lim} as a statement about the universality of the self-similar solution $g_t(x)$ for the large-time behavior of all solutions to \eqref{c.h.eq}.

Theorem~\ref{lim norm th}
gives an analogous result for $u(t,x)=\tilde P_t f(x)$ and the initial value problem
\begin{equation}\label{equation}
	\begin{cases}
		\partial_t u(x,t)=\left(\Delta^{\alpha/2} +\kappa  |x|^{-\alpha}\right) u(x,t) , \quad x\in \bbfR^d, \quad t>0 ,\\
		u(x,0)=f(x) .
	\end{cases}
\end{equation}  
for $\kappa$ satisfying \eqref{e.cdk} and 
$f\in L^1(H)$. 
The use of $h$ is novel in this setting -- this
and connections to the classical literature will be discussed in more detail in Section~\ref{s.ssa}.

\subsection{General conventions} \label{s:n}
We tend to use ``:='' to indicate definitions, e.g., $a\land b := \min\{a, b\}$ and $a\vee b := \max\{a, b\}$, $a_+:=a \vee 0$ and $a_-:=(-a)\vee 0$. Throughout, we only consider  Borel measurable  functions and Borel measures. 
As usual, integrals are considered well-defined when the integrands are nonnegative or absolutely integrable with respect to a given measure. In the case of integral kernels, the corresponding integrals
should at least be well-defined pointwise almost everywhere ($a.e.$).
For $x\in \Rd$ and $r>0$ we define $B(x,r)=\{y\in \Rd: |y-x|<r\}$, the ball with center at $x$ and radius $r$.
We write $f\approx g$, which we call \textit{approximation} or \textit{comparison}, and say $f$ and $g$ are \textit{comparable}, if $f,g$ are nonnegative functions, $c^{-1}g\le f\le cg$ with some constant $c$, that is a number in $(0,\infty)$. The values of constants may change without notice from line to line in a chain of estimates. Of course, we shall also use constants in inequalities (one-sided comparisons of functions), e.g., $f\le cg$. We occasionally write $c=c(a,\ldots,z)$ to assert that the constant $c$ may be so selected as to depend only on $a,\ldots,z$.
As usual, for $1\le p\le \infty$, $L^p:=L^p(\Rd, \ud x)$, with norm $\|\cdot \|_p$, and 
$L^p(g):=L^p(\Rd, g \ud x)$  with  norm $\| f \|_{L^p(g)}$ and nonnegative (weight) function $g$.
\subsection*{Acknowledgements}
We thank Tomasz Grzywny, Piotr Knosalla, Tomasz Komorowski, Alex Kulik, Markus Kunze, \L{}ukasz Le\.zaj, Grzegorz Serafin, \L{}ukasz Stettner, 
and Tomasz Szarek
for helpful discussions, comments and references. 
We particularly thank Tomasz Komorowski for much advice on the proof of Theorem~\ref{c.gsall} and Alex Kulik for discussions on \eqref{e.convL}.
\section{Preliminaries}
\subsection{Fractional Laplacian}
Let 
$$
\nu(y)
=\frac{ \alpha 2^{\alpha-1}\Gamma\big((d+\alpha)/2\big)}{\pi^{d/2}\Gamma(1-\alpha/2)}|y|^{-d-\alpha}\,,\quad y\in \Rd\, .
$$
The 
coefficient  is so chosen 
that 
\begin{equation}
  \label{eq:trf}
  \int_{\Rd} \left[1-\cos(\xi\cdot y)\right]\nu(y)\ud y=|\xi|^\alpha\,,\quad
  \xi\in \Rd\,,
\end{equation}
see, e.g., Bogdan, Byczkowski, Kulczycki, Ryznar, Song, and Vondra{\v{c}}ek \cite[(1.28)]{MR2569321}. The fractional Laplacian for (smooth compactly supported) \textit{test functions} $\varphi\in C^\infty_c(\Rd)$ is
\begin{equation*}
  \Delta^{\alpha/2}\varphi(x) = 
  \lim_{\varepsilon \downarrow 0}\int_{|y|>\varepsilon}
  \left[\varphi(x+y)-\varphi(x)\right]\nu(y)\ud y\,,
  \quad
  x\in \Rd\,.
\end{equation*}
Many authors use the notation $-(-\Delta)^{\alpha/2}$ for the operator. In terms of the Fourier transform, 
$\widehat{\Delta^{\alpha/2}\varphi}(\xi)=-|\xi|^{\alpha}
\widehat{\varphi}(\xi)$, see, e.g., \cite[Section~1.1.2]{MR2569321} or \cite{MR3613319}.

\subsection{The semigroup of $\Delta^{\alpha/2}$}
We consider the convolution semigroup of functions
\begin{equation}
  \label{eq:dpt}
  p_t(x):=(2\pi)^{-d}\int_ \Rd e^{-t|\xi|^\alpha}e^{-ix\cdot\xi}\ud \xi\,,\quad\ t>0,\ x\in
  \Rd\,.
\end{equation}
According to 
(\ref{eq:trf}) 
and the L\'evy-Khinchine formula,
each $p_t$ is a radial probability density function and $\nu(y)\ud y$ is the L\'evy measure of the semigroup, 
see, e.g., 
\cite{MR2569321}.
From (\ref{eq:dpt}) we have 
\begin{equation}
  \label{eq:sca}
  p_t(x)=t^{-d/\alpha}p_1(t^{-{1/\alpha}}x)\,.
\end{equation}
It is well-known that 
$p_1(x)
\approx1\land |x|^{-d-\alpha}$ (see   \cite{MR2569321}, Bogdan, Grzywny and Ryznar \cite[remarks after Theorem~21]{MR3165234} or \cite{MR3613319}), so
\begin{equation}\label{eq:oppt}
p_t(x)
\approx t^{-d/\alpha}\land \frac{t}{|x|^{d+\alpha}}
\,,\quad t>0,\ \ x\in \Rd\,.
\end{equation}
Since $\alpha<d$, we get (the Riesz kernel) 
\begin{align}\label{eq:pot}
	\int_0^\infty p_t(x)\ud{t} = \Aa_{d,\alpha} |x|^{\alpha-d}, \quad x\in \Rd,
\end{align}
where
\begin{equation}\label{e.dAda}
\Aa_{d,\alpha} = \frac{\Gamma(\frac{d-\alpha}{2})}{\Gamma(\frac{\alpha}{2})2^\alpha \pi^{d/2}},
\end{equation}
see, e.g., \cite[Section 1.1.2]{MR2569321}.
We denote 
$$
p(t,x,y)=p_t(y-x), \quad t>0,\ x,y\in \Rd.
$$
Clearly, $p$ is symmetric:
\begin{equation*}\label{pt}
   p(t,x, y) = p(t,y, x), \quad t>0,\ x,y\in \Rd,
\end{equation*}   
and satisfies the Chapman-Kolmogorov equations:
\begin{equation}\label{eq:ChKforp}
     \int_{\Rd} p(s,x, y)p(t,y, z)\ud y = p(t+s,x, z), \quad x, z \in  \Rd,\, s, t > 0.
\end{equation}     
We denote, as usual, $P_t g(x) = \int_{\Rd} p(t,x, y) g(y)\ud y$.
The fractional Laplacian extends to the generator of the semigroup $\{P_t\}_{t>0}$ on many Banach spaces, see, e.g., \cite{MR3613319}.

\subsection{Schr\"odinger perturbation by Hardy potential}
We recall 
elements of the integral analysis of \cite{MR3460023} and \cite{MR3933622}, which was used to handle the heat kernel  $\tilde p$ of 
$\Delta^{\alpha/2}+\kappa |x|^{-\alpha}$. Thus, 
$$f_\beta(t) := c_\beta t_+^{(d-\alpha -\beta)/\alpha },\quad t\in \R,$$ 
for $\beta \in (0, d)$. The constant $c_\beta$ is so chosen that 
\begin{equation}\label{h_beta}
   h_\beta(x) := \int_0^\infty f_\beta(t) p_t(x) {\rm d}t =|x|^{-\beta}, \quad x\in \Rd.
\end{equation}
The existence of such $c_\beta\in (0, \infty)$ follows from \eqref{eq:sca} and the estimate $p_1(x)
\approx1\land |x|^{-d-\alpha}$. Of course, $f_\beta'(t)=c_\beta\frac{d-\alpha-\beta}{\al}t^{(d-2\alpha -\beta)/\alpha }$. Accordingly, for $\beta \in (0, d-\alpha)$ we may define
\begin{equation*}\label{def_q}
     q_\beta(x) = \frac{1}{h_\beta(x)}\int_0^{\infty }f_\beta'(t) p_t(x) {\rm d}t,\quad x\in \R_0^d,
\end{equation*}
where $\R_0^d :=\Rd\setminus \{0\}$. By \cite[(26)]{MR3460023}\footnote{The exponent $(d-\alpha -\beta)/\alpha$ in the definition of $f$ is denoted $\beta$ in \cite[Corollary~6]{MR3460023}.},
\begin{equation*}\label{q_beta}
  q_\beta(x) = \kappa_\beta |x|^{-\alpha },
\end{equation*}
where $\kappa_\beta$ is defined by \eqref{e.dkd}.
In what follows we keep our notation from \eqref{e.cdk}, that is, we let \begin{equation}\label{eq:dk0}
\delta \in [0,(d-\alpha)/2], \quad 	
\kappa = \kappa_\delta,\quad
h(x)=h_\delta(x)=|x|^{-\delta},\quad 
q(x) = q_\delta(x)=\kappa |x|^{-\alpha}. 
\end{equation}
To wit, the case of $\delta \in (0,(d-\alpha)/2]$ is covered by the discussion of $\beta$ above, and $\delta=0$ yields the trivial $\kappa=0$, $q=0$ and $h=1$.
We then define the Schr\"odinger perturbation of $p$ by $q$:
\begin{equation}\label{def_p_tilde}
   \tilde p =\tilde p_\delta= \sum_{n=0}^{\infty }p_n.
\end{equation}
Here for $t>0$ and $x,y\in \Rd$ we let $p_0(t,x, y) = p(t,x, y)$ and then proceed by induction:
\begin{align}\label{pn}
     p_n(t,x, y) &= \int_0^t \int_{\Rd} p(s,x, z) q(z)p_{n-1}(t-s, z, y)   \ud z    \ud s \\
     &=\int_0^t \int_{\Rd} p_{n-1}(s,x, z) q(z)p(t-s, z, y)   \ud z    \ud s 
     , \quad n \geq 1.\notag
\end{align}    
Of course, $\tp_0=p$. 
By \eqref{eq:oppt}, for $t>0$ and $y\in \Rd$ we have
\begin{align*}
p_1(t,0,y)&=\int_0^t \int_{\Rd} p(s,0, z) q(z)p(t-s, z, y)   \ud z    \ud s \\
&\ge c_{t,y}\int_0^{t/2} \int_{|z|<s^{1/\al}} s^{-d/\al} |z|^{-\al}  \ud z    \ud s =\infty.
\end{align*}
By symmetry, $p_1(t,x,0)=\infty$, too, for all $x\in \Rd$ and $t>0$, therefore $\tp(t,x,y)=\infty$ if $x=0$ or $y=0$.
By \cite[Theorem 1.1]{MR3933622}, the above discussion of $x=0$ and $y=0$ and the usual notational conventions we have for all $x,y\in \Rd$, $t>0$,
\begin{equation}
\label{eq:mainThmEstz}
\tilde{p}(t,x,y)\approx \left( t^{-d/\alpha}\wedge \frac{t}{|x-y|^{d+\alpha}} \right)\left(1+t^{\delta/\alpha}|x|^{-\delta}  \right)\left(1+t^{\delta/\alpha}|y|^{-\delta}  \right).
\end{equation}
Clearly, if $0<t_1<t_2<\infty$, then
\begin{equation}\label{e.crt1t2p}
\tilde p(s,x,y)\approx \tilde p(t, x,y), \quad x,y\in \Rd, \quad t_1\le s,t\le t_2
\end{equation}
(the  comparability 
constant does depend on $t_1$, $t_2$ and $d$, $\alpha$, $\delta$).
Recall that
\begin{equation}\label{e.dH}
	H(x)=1\vee h(x)=1\vee |x|^{-\delta}\approx 1+|x|^{-\delta},\quad x\in \Rd.
\end{equation}
Thus, we can reformulate \eqref{eq:mainThmEstz} as follows,
\begin{equation}
	\label{eq:mainThmEstz2}
	\tilde{p}(t,x,y)\approx p(t,x,y)\ H(t^{-1/\alpha}x)\  H(t^{-1/\alpha}y), \qquad t>0,\quad x,y\in \Rd.
\end{equation}
By \cite{MR3933622} and the above conventions,
 $\tilde p$ is a symmetric time-homogeneous transition density on $\Rd$, in particular the Chapman-Kolmogorov equations hold:
\begin{align}\label{eq:ChK}
\int_\Rd \tp(s,x,z) \tp(t,z,y) \ud   z  = \tp(t+s,x,y)\,,\quad {x,y\in \Rd, \ s,t>0.}
\end{align} 
The following Duhamel formulae hold for $p$ and $\tilde p$,
\begin{align}\label{eq:Df1}
\tp(t,x,y) &= p(t,x,y) + \int_0^t \int_{\bbfR^d} p(s,x,z) q(z) \tp(t-s,z,y) \ud   z   \ud s \\
&= p(t,x,y) + \int_0^t \int_{\bbfR^d} \tp(s,x,z) q(z) p(t-s,z,y)  \ud z   \ud s ,\quad t>0,\ x,y\in \bbfR^d. \label{eq:Df2}
\end{align}
In passing we refer to Bogdan, Hansen and Jakubowski \cite{MR2457489} and Bogdan, Jakubowski and Sydor \cite{MR3000465} for a general setting of Schr\"odinger perturbations of transition semigroups and other families of integral kernels. 

The function $\tilde p$ is self-similar, i.e., has the following scaling \cite[Lemma~2.2]{MR3933622}:
\begin{equation}\label{ss form}
  \tilde p(t,x, y)=t^{-d/\alpha } \tilde p\big(1, t^{-1/\alpha }x , t^{-1/\alpha }y \big), \qquad  t>0,  \ x,y \in \bbfR^d.
\end{equation}
This is the same scaling as for $p$.
Furthermore, if $T$ is a (linear) isometry of $\Rd$, then for all $t>0$, $x,y \in \Rd$ we have $p(t,Tx,Ty)=p_t(T(y-x))=p_t(y-x)=p(t,x,y)$, because  $p_t$ is radial. Of course, $q$ is radial, too.
By the change of variables $z=Tv$ and induction, 
\begin{align*}
     &p_n(t,Tx, Ty) = \int_0^t \int_{\Rd} p(s,Tx, z) q(z)p_{n-1}(t-s, z, Ty)   \ud z    \ud s \\
       &= \int_0^t \int_{\Rd} p(s,Tx, Tv) q(Tv)p_{n-1}(t-s, Tv, Ty)   \ud v   \ud s 
   \\    &=\int_0^t \int_{\Rd} p(s,x, v) q(v)p_{n-1}(t-s, v, y)   \ud v   \ud s =
       p_n(t,x, y),
     \quad n \geq 1.
\end{align*}
Therefore,
\begin{equation}\label{isom}
  \tilde p(t,Tx, Ty)=\tp (t,x,y),
\qquad  t>0,  \ x,y \in \bbfR^d.
\end{equation}

\subsection{Doob's conditioning}
Recall that $\delta \in [0,(d-\alpha)/2]$ and $\kappa = \kappa_\delta$, $h(x)=h_\delta (x)=|x|^{-\delta}$, $\tilde p$ depends on $\delta$. 
By \cite[Theorem~3.1]{MR3933622} and the preceding discussions, the function $h$ is \textit{invariant} in the following sense:
\begin{align}\label{eq:2}
	\int_{\bbfR^d} \tp(t,x,y)  h(y)
\ud y = h(x),
\qquad t>0,\; x \in \Rd\,.
\end{align}
We define the following Doob-conditioned (renormalized) kernel
\begin{equation}\label{e:drhot}
\rho_t(x,y)=\frac{\tilde p(t,x,y)}{h(x)h(y)},\quad t>0,\quad x,y\in \Rdz
\end{equation}
(later on we shall extend $\rho$ to $(0,\infty)\times\Rd\times\Rd$). 
We consider the integral weight $h^2(x),  x\in \Rd$.
{Doob-type conditioning is also called Davies' method, see Murugan and Saloff-Coste \cite{MR3601569}, and $h$ is sometimes called desingularizing weight, see Milman and Norsemen \cite{MR2064932}.}
By \eqref{eq:2},
\begin{equation}\label{e:rhois1}
\int_\Rd \rho_t(x,y)h^2(y)\ud y
=1, \quad x\in \Rdz, \quad t>0.
\end{equation}
By \rf{eq:ChKforp},
\begin{align}
     &\int_{\Rd} \rho_s(x, y)\rho_t(y, z)h^2(y)\ud y = 
     \int_{\Rd} \frac{\tilde p(s,x, y)}{h(x)h(y)}\frac{\tilde p(t,y, z)}{h(y)h(z)}h^2(y)\ud y\nonumber\\
     &=\frac{1}{h(x)h(z)} \tilde p(t+s,x, z)
     =\rho_{t+s}(x,z), \quad x, z \in  \Rdz,\, s, t > 0.\label{e.CKfrho}
\end{align} 
We see that $\rho$ is a symmetric time-homogeneous transition probability density on $\Rdz$ with the reference measure $h^2(y)\ud y$. 
For nonnegative $f\in L^1(h^2)$, 
by Fubini-Tonelli and \eqref{e:rhois1},
\begin{align*}
\int_\Rd \int_\Rd f(y)\rho_t(x,y)h^2(y) \ud y \,h^2(x) \ud x 
 &=\int_\Rd f(y) h^2(y) \ud y ,
\end{align*}
so the operators
$$
\mathcal{R}_t f(x):=\int_\Rd f(y)\rho_t(x,y)\,h^2(y)  \ud y , \quad t>0, 
$$
are contractions on $L^1(h^2)$.
By \eqref{ss form}, $\rho$ is self-similar: for $t>0$ and $x,y \in \bbfR^d$ we have
\begin{equation}\label{ss form rho}
  \rho_t(x, y)=\frac{t^{-d/\alpha } \tilde p\big(1, t^{-1/\alpha }x , t^{-1/\alpha }y \big)}
  {t^{-\delta/\alpha}h(t^{-1/\alpha}x)t^{-\delta/\alpha}h(t^{-1/\alpha}y)}
  =t^{\frac{2\delta-d}{\alpha} } \rho_1\big(t^{-1/\alpha }x , t^{-1/\alpha }y \big),
\end{equation}
hence
\begin{equation}\label{ss form rhos}
  \rho_{st}(t^{1/\alpha }x, t^{1/\alpha }y)
  =t^{\frac{2\delta-d}{\alpha} } \rho_s(x , y), \quad s>0.
\end{equation}
For each (linear) isometry $T$ of $\Rd$,
$h\circ T=h$, thus by \eqref{isom} we get
\begin{equation}\label{risom}
  \rho_t(Tx, Ty)=\rho_t (t,x,y),
\qquad  t>0,  \ x,y \in \Rdz.
\end{equation}
By \eqref{eq:mainThmEstz},
\begin{equation}
\label{eq:mainThmEstrho}
\rho_1(x,y)\approx \left(1 \wedge {|x-y|^{-d-\alpha}} \right)\left(1+|x|^{\delta}  \right)\left(1+|y|^{\delta}  \right) ,\quad 
 \ x,y \in \Rdz.
\end{equation}
Notably, $\rho_1$ is not bounded (consider large $x=y$). By \eqref{ss form rho},
\begin{equation}
\label{eq:mainThmEstrhot}
\rho_t(x,y)\approx \left( t^{-d/\alpha}\wedge \frac{t}{|x-y|^{d+\alpha}} \right)\left(t^{\delta/\alpha}+|x|^{\delta}  \right)\left(t^{\delta/\alpha}+|y|^{\delta}  \right),\quad 
t>0,  \ x,y \in \Rdz.
\end{equation}
We also note that if $0<t_1<t_2<\infty$, then
\begin{equation}\label{e.crt1t2}
\rho_{s}(x,y)\approx \rho_{t}(x,y), \quad x,y\in \Rdz, \quad t_1\le s,t\le t_2
\end{equation}
(the  comparability constant depends on $t_1$, $t_2$ and $d$, $\alpha$, $\delta$).
By \eqref{eq:mainThmEstrho}, for every $M\in (0,\infty)$,
\begin{equation}
\label{eq:mainThmEstrho2}
\rho_1(x,y)\approx (1+|y|)^{-d-\alpha+\delta}, \quad 0<|x|\le M,\quad y\in \Rdz.
\end{equation}
Clearly, the right-hand side is integrable with respect to $h^2(y) \ud y $ and bounded. Our aim is to prove the continuity of $\rho_1$, notably at $x=0$.

\section{Limiting behaviour at the origin
}\label{s:csss}
Here is a full-fledged variant of Theorem~\ref{t:eta2}.
\begin{theorem}\label{t:eta}
The function $\rho$ has a continuous extension to $(0,\infty)\times \Rd\times \Rd$ and 
\begin{equation}\label{e.det}
\rho_t(0,y):=\lim_{x\to 0}\rho_t(x,y),
\quad t>0,\quad y\in \Rdz ,
\end{equation}
satisfies:
\begin{equation}\label{ss form eta}
	\rho_t(0,y)
	= t^{\frac{2\delta-d}{\alpha} } \rho_1(0,t^{-1/\alpha }y), \quad t>0,\quad y\in \Rdz ,
\end{equation}
\begin{equation}\label{e:ssf}
\int_\Rd \rho_t(0,y)\rho_s(y,z)h^2(y)\ud y=\rho_{t+s}(0,z),\quad s,t>0,\quad z\in \Rdz .
\end{equation}
\end{theorem}
The proof of Theorem~\ref{t:eta} is given below in this section. Let us explain the line of attack.
Given \eqref{e.det} and taking the limit in \eqref{e.CKfrho} as $x\to 0$, 
we should get \eqref{e:ssf}.
By \eqref{ss form rho} we should then obtain \eqref{ss form eta}
and
\begin{equation}\label{e:ssf2}
\int_\Rd t^{\frac{2\delta-d}{\alpha}}\rho_1(0,t^{-1/\alpha}y)\rho_s(y,z)h^2(y)\ud y=(t+s)^{\frac{2\delta-d}{\alpha}} \rho_1(0,(t+s)^{-1/\alpha}z),\quad 
s,t>0, z \in \Rdz.
\end{equation}
Changing variables $u=(t+s)^{-1/\alpha}z$ and $x=t^{-1/\alpha}y$, we see that \eqref{e:ssf2} is equivalent to
\begin{equation}\label{e:ssf3}
\int_\Rd \rho_1(0,x)(t+s)^{\frac{d-2\delta}{\alpha}}\rho_s(t^{1/\alpha}x,(t+s)^{1/\alpha}u)h^2(x)\ud x= 
\rho_1(0,u).
\end{equation}
By 
\eqref{ss form rhos}, this is the same as
\begin{equation}\label{e:ssf3b}
\int\limits_\Rd \rho_{\frac{s}{s+t}}\left(\left(t/(s+t)\right)^{1/\alpha}x,u\right)\rho_1(0,x)h^2(x)\ud x= \rho_1(0,u).
\end{equation}
In what follows we shall \textit{define} $\rho_1(0,\cdot)$ as a solution to the integral equation \rf{e:ssf3b} and then essentially reverse the above reasoning. Additionally, to simplify the notation and arguments
we introduce an auxiliary Ornstein-Uhlenbeck-type semigroup. 
\subsection{Ornstein-Uhlenbeck semigroup}
For $f\ge 0$  we let
\begin{equation}
L_tf(y)=\int_\Rd l_t(x,y) f(x)h^2(x) \ud x , \label{e:ssf4e}
\end{equation}
where
\begin{equation}\label{e.dlt}
l_t(x,y)=\rho_{1-e^{-t}}(e^{-t/\alpha}x,y), \qquad t>0,\quad x,y\in \Rdz.
\end{equation}
By \eqref{e.CKfrho} and \eqref{ss form rhos}, 
\begin{align*}
&\int_\Rd l_s(x,y) l_t(y,z)h^2(y)\ud y=
\int_\Rd \rho_{1-e^{-s}}(e^{-s/\alpha}x,y)  \rho_{1-e^{-t}}(e^{-t/\alpha}y,z)h^2(y) \ud y \\
&=\int_\Rd \rho_{1-e^{-s}}(e^{-s/\alpha}x,y)  
(e^{-t})^{\frac{2\delta-d}{\alpha}} \rho_{e^{t}-1}(y,e^{t/\alpha}z)h^2(y) \ud y \\
&=(e^{-t})^{\frac{2\delta-d}{\alpha}} 
\rho_{e^t-e^{-s}}(e^{-s/\alpha}x,  e^{t/\alpha} z)
=\rho_{1-e^{-s-t}}(e^{-(s+t)/\alpha}x,z) =l_{t+s}(x,z).
\end{align*}
Thus $l_t(x,y)$ is a transition density on $\Rdz$ 
with respect to the measure $h^2(y) \ud y $. By Fubini's theorem, 
$\{L_t\}_{ t>0}$ 
 is a semigroup of operators on $L^1(h^2)$, an Ornstein-Uhlenbeck-type semigroup, see \cite[solution to E 18.17 on p. 462--463]{MR3185174}. It shall be a major technical tool in our development.

\subsection{Stationary density}

If  $\varphi\ge 0$
and $\int \varphi(x)h^2(x)\ud x=1$, then we say that $\varphi$ is a \textit{density}.
By Fubini-Tonelli Theorem and \eqref{e:rhois1}, for $t>0$ and $f\ge0$, we have
\begin{align}\label{e.iso-1}
\int_\Rd L_t f(y)h^2(y) \ud y  
&=\int_\Rd f(x)h^2(x) \ud x .
\end{align}
Thus, the operators $L_t$ preserve densities. So they are Markov, see
Komorowski \cite{MR1162571}, 
Lasota and Mackey \cite{MR1244104}, Lasota and York \cite{MR1265226} for this setting.

We say that a density $\varphi$ is \textit{stationary} for
$L_t$ if
$
L_t\varphi=\varphi.
$

\begin{theorem}\label{c.gsall} There is a unique stationary  density $\varphi$ for  the operators 
$L_t$, $t>0$.
\end{theorem}
\begin{proof}
Fix $t>0$ and let $P=L_t$ 
so that 
$
	P^k=L_{kt}
$, $k=1,2,\ldots$.
By \eqref{e.dlt} and \eqref{ss form rho},  for all $f\ge 0$ and $k\in \bbfN$ 
we have
\begin{eqnarray}\label{e.pts}
P^{k}f(u) &=& \int_{\bbfR^d}  f(y)
\rho_{1-e^{-kt}}(e^{-kt/\alpha}y,u) h^2(y) \ud y \\
&=& \int_{\bbfR^d}  f(y)
(e^{-kt})^{\frac{2\delta-d}{\alpha}} \rho_{e^{kt}-1}(y,e^{kt/\alpha}u) h^2(y) \ud y , \qquad u \in \Rdz.\nonumber
\end{eqnarray}
Let $B=\{x\in \Rd: 0<|x|\le 1\}$. We write $f\in F$ if
\begin{equation}\label{e.comp}
f(y)=\int_B \rho_1(x,y)\mu(\!\ud x)
\end{equation} 
for some subprobability measure $\mu$ concentrated on $B$. Then,
by  \eqref{e.CKfrho} and \eqref{ss form rho},
\begin{align}\nonumber
P^{k}f(u) &= \int_B  (e^{-kt})^{\frac{2\delta-d}{\alpha}}
\int_{\bbfR^d}  \rho_1(x,y) \rho_{e^{kt}-1}(y,e^{kt/\alpha}u) h^2(y) \ud y \mu(\!\ud x)\\
\nonumber
&= \int_B    (e^{-kt})^{\frac{2\delta-d}{\alpha}}
 \rho_{e^{kt}}(x,e^{kt/\alpha}u)\mu(\!\ud x)\\
\nonumber
&= \int_B  \rho_1\left(e^{-kt/\alpha}x, u\right)\mu(\!\ud x)\\
\label{e.pts3}
&= \int_B  \rho_1(x, u)\tilde \mu(\!\ud x)\approx (1+|u|)^{-d-\alpha+\delta}, \qquad u \in \Rdz, 
\end{align} 
where $\tilde \mu$ is a subprobability measure concentrated on $e^{-kt/\alpha}B\subset B$. Thus, $P^k F\subset F$. We note that the comparison \eqref{e.pts3} is independent of $k$, see \eqref{eq:mainThmEstrho2}.
We next argue
that the hypotheses of \cite[Theorem~3.1]{MR1162571} hold true for $P$ and $F$: (C1) -- the lower bound, and (C2) -- the uniform absolute continuity. Indeed, by \eqref{e.pts3},
$$\liminf_{n\to \infty}\frac1n\sum_{k=1}^n
\int_B P^k f(u) h^2(u)\ud u
\approx \int_B (1+|u|)^{-d-\alpha+
\delta} h^2(u)\ud u
>0,$$ which yields (C1) in \cite{MR1162571}.
Also, (C2) therein is satisfied because
$$\int_A P^kf(y)h^2(y) \ud y \le C\int_A (1+|y|)^{-d-\alpha+
\delta}h^2(y) \ud y \to 0,$$
uniformly in $k$ as $\int_A h^2(y) \ud y\to 0$, due to the integrability of $(1+|y|)^{-d-\alpha+2\delta}$.
By \cite[Theorem~3.1]{MR1162571}, a density $\varphi$ exists satisfying $L_t\varphi=P\varphi=\varphi$.
Moreover, the 
density $\varphi$ is unique. Indeed, if $\psi$ is a probability density with respect to $h^2(x)\ud x$ and $P\psi=\psi$, then 
$r:=\varphi-\psi$ satisfies $Pr=r$, too. If $r=0$ $a.e.$, then we are done. Otherwise, $\int_\Rd r_+(x) h^2(x) \ud x =\int_\Rd r_-(x) h^2(x) \ud x >0$. Then, on the one hand $|Pr|=|r|$, on the other hand for $a.e.$ $x\in \Rd$,
\begin{align*}
Pr(x)&=P r_+(x)-Pr_-(x)\\
&= \int_{\bbfR^d}  r_+(y) 
l_t(y,x) h^2(y) \ud y 
-\int_{\bbfR^d}  r_-(y) 
l_t(y,x) h^2(y) \ud y 
\end{align*}
and both terms are nonzero, because 
$l_t$ is 
positive. Therefore ,
$$|Pr(x)|<P r_+(x)\vee Pr_-(x)\le P|r|(x).$$
By this and \eqref{e.iso-1},
\begin{align*}
\int_\Rd |r(x)|h^2(x) \ud x &=\int_\Rd |Pr(x)|h^2(x) \ud x <\int_\Rd P|r|(x)h^2(x) \ud x =\int_\Rd 
|r(x)| h^2(x) \ud x .
\end{align*}
We obtain a contradiction, so $r=0$, and $\psi=\varphi$ in $L^1(h^2)$. In passing we note that the above argument is a part of the proof of Doob's theorem \cite[Theorem~4.2.1]{MR1417491}.

Since the operators $L_t$, $t>0$, commute, they have the same stationary density, that is $\varphi$. Indeed, if $s>0$ and $Q=L_s$, then $P(Q\varphi)=QP\varphi=(Q\varphi)$, and $Q\varphi$ is a density, so by the uniqueness, $Q\varphi=\varphi$.
\end{proof}
In passing we like to refer the interested reader to additional literature on the existence and uniqueness of stationary densities and measures. The book of Foguel \cite{MR0261686} gives a concise introduction to ergodic theory of Markov processes, in particular to the $L^1$ setting. Stationary measures and densities are discussed in 
Da Prato and Zabczyk \cite[Remark~3.1.3]{MR1417491}, \cite[Theorem~3.1]{MR1265226}, Lasota \cite[Theorem~6.1]{MR1452617}, and Komorowski, Peszat, Szarek \cite{MR2663632}. 
A nice presentation of asymptotic stability (and periodicity) of Markov operators is given by Komornik \cite{MR826761}, see also Lasota and Mackey \cite[p. 373, 11.9.4]{MR1244104}, Komorowski and Tyrcha \cite{MR1101473}, and Stettner \cite{MR1810695}.

We note that the existence of a stationary density in our setting would also follow from the -- perhaps more constructive -- approach using the weak compactness of the set of functions given by \eqref{e.comp} and Schauder-Tychonoff fixed point theorem, see, e.g., Rudin \cite[5.28 Theorem]{MR1157815}, see also \cite[Section~3.8]{MR1157815}, the Dunford-Pettis theorem in Voigt \cite[Chapter~15]{MR4182424} and \eqref{e.pts3} above. We could also use the Krylov-Bogolioubov theorem, see, e.g., \cite{MR1244104} or the minicourse of Hairer \cite{mc2021MH}; and it is always worthy to check with Doob \cite{MR25097}, especially that his example (0.9) touches upon the classical Ornstein-Uhlenbeck semigroup.

In what follows, $\varphi$ denotes the stationary  density  for  the operators $L_t$, $t>0$
(in Lemma 
\ref{l.ext}
 below we verify that $\varphi$ can be defined pointwise so as to be continuous).

\subsection{Asymptotic stability}
The state space $\Rdz $ is a Polish space and $l_t$ is positive. By Theorem~\ref{c.gsall} and the results 
of Kulik and Scheutzov \cite[Theorem~1 and Remark~2]{MR3345324}, for every $x\in \Rdz $ we get
\begin{equation}\label{e.convL}
\int_\Rd \left|l_t(x,y) -\varphi(y)\right|h^2(y) \ud y \to 0 \mbox{ as } t\to \infty.
\end{equation}
This is the (large-time) asymptotic stability of the semigroup aforementioned in the title of this subsection.
Let $A$ be a bounded subset of $\Rdz$ and $x,x_0\in A$.
By Theorem~\ref{c.gsall} and \rf{eq:mainThmEstrho},
\begin{align}
\int_\Rd \left|l_{1+t}(x,y) -\varphi(y)\right|h^2(y) \ud y 
&=\int_\Rd  \left| \int_\Rd l_1(x,z) \left(l_{t}(z,y) -\varphi(y)\right) h^2(z) \ud   z  \right|h^2(y) \ud y \nonumber \\
&
\le 
c\int_\Rd  l_1(x_0,z) \int_\Rd  \left|l_{t}(z,y) -\varphi(y)\right| h^2(y) \ud y\  h^2(z)\ud z.\label{e.li}
\end{align}
We have $I_t(z):=\int_\Rd  \left|l_{t}(z,y) -\varphi(y)\right| h^2(y) \ud   y  \to 0$ as $t\to \infty$. Furthermore,  for every $z\in \Rdz$,
$I_t(z)\le \int_\Rd  \left(l_{t}(z,y) +\varphi(y)\right) h^2(y) \ud y=2$.
Of course, $\int_\Rd  2 l_1(x_0,z) \ h^2(z)\ud z=2<\infty$.
By the dominated convergence theorem, the iterated integral in \eqref{e.li} converges to $0$, therefore the convergence in \eqref{e.convL} is uniform for $x\in A$. In terms of $\rho$, \eqref{e.convL} reads as follows: uniformly in $x\in A$,
\begin{equation}\label{e.conv}
\int_\Rd \left|\rho_{1-e^{-t}}(e^{-t/\alpha}x,y) -\varphi(y)\right|h^2(y) \ud y \to 0 \mbox{ as } t\to \infty.
\end{equation}
As a consequence we obtain the following spatial convergence in $L^1(h^2)$.
\begin{lemma}\label{l.crho}
We have $\int_\Rd \left|\rho_{1}(x,y) -\varphi(y)\right|h^2(y) \ud y \to 0$ as $x\to 0$.
\end{lemma}
\begin{proof}
It suffices to make  cosmetic changes to \eqref{e.conv}. Of course, $e^{-t/\alpha}x\to 0$ and $1-e^{-t}\to 1$ as $t\to \infty$. By scaling, 
$$\rho_{1-e^{-t}}(e^{-t/\alpha}x,y)=(1-e^{-t})^{(2\delta-d)/\alpha}\rho_1\big((e^t-1)^{-1/\alpha}x,(1-e^{-t})^{-1/\alpha}y\big),$$ 
so
\begin{align}\label{rho_et}
\int_\Rd \left|\rho_1\big((e^t-1)^{-1/\alpha}x,(1-e^{-t})^{-1/\alpha}y\big) -\varphi(y)\right|h^2(y) \ud y \to 0 \mbox{ as } t\to \infty.
\end{align}
By the continuity of dilations on $L^1(\!\ud x)$,
\begin{align*}
\int_\Rd \left|\varphi\big((1-e^{-t})^{-1/\alpha}y\big) h^2(y)-\varphi(y)h^2(y)\right| \ud y \to 0 \mbox{ as } t\to \infty.
\end{align*}
Using the triangle inequality and changing variables in \rf{rho_et} we get
\begin{align*}
	\int_\Rd \left|\rho_1\big((e^{t}-1)^{-1/\alpha}z,y\big) -\varphi(y)\right|h^2(y) \ud y \to 0 \mbox{ as } t\to \infty
\end{align*}
uniformly in $z\in A$ for bounded $A\subset \Rdz$. We take $A$ as the unit sphere, for $x\in \Rdz$ we write $x=(e^t-1)^{-1/\alpha}z$, where $t=\ln (1+|x|^{-\alpha})$ and $z=x/|x|\in A$, and we obtain the result.
\end{proof}
By Lemma~\ref{l.crho} and \eqref{eq:mainThmEstrho2}
we obtain the following estimate.
\begin{corollary}\label{c.bphi}
$\varphi(y)\approx (1+|y|)^{-d-\alpha+\delta}$ for almost all $y\in \Rd$.
\end{corollary}
By the next result we may actually consider  $\vf$ as defined pointwise.
\begin{lemma}\label{l.ext}
After modification on a set of Lebesgue measure zero, $\varphi$ is a continuous radial function on $\Rd$ and $\varphi(y)\approx (1+|y|)^{-d-\alpha+\delta}$ for all $y\in \Rd$.
\end{lemma}
\begin{proof}By Theorem~\ref{c.gsall}, $\vf=L_1\vf$ ($a.e.$), so  it suffices to prove that
$L_1 \vf$ is $a.e.$ equal to a continuous function on $\Rd$. 
By \eqref{e.crt1t2} and \eqref{eq:mainThmEstrho2}, $c(1+|y|)^{-d-\alpha+\delta}$ is an integrable majorant of $l_1(x,y)$ for bounded $x$. Furthermore, $\vf$ is essentially bounded ($a.e.$), by Corollary~\ref{c.bphi}.
The function $l_1(x,y)$ is continuous in $y\in \Rdz$ so, by the dominated convergence theorem, $L_1\vf(y)$ is continuous for $y\in \Rdz$.
It remains to prove
the convergence of
$L_1\vf(y)$ to a finite limit as $y\to 0$.
We let $t>0$ and $y \to 0$. By scaling, changing variables, the symmetry of $\rho_1$ and Lemma~\ref{l.crho}, 
\begin{align}
\nonumber
L_t \vf(y)&=\int_\Rdz \rho_{1-e^{-t}}\big(e^{-t/\alpha}x,y\big)\vf(x)h^2(x) \ud x \\\nonumber
&=\int_\Rdz e^{t(d-2\delta)/\alpha}\rho_{1}\big(z, (1-e^{-t/\alpha})^{1/\alpha}y\big)\vf\big((e^t-1)^{1/\alpha}z\big)h^2(z) \ud z \\\nonumber
&\to
\int_\Rdz e^{t(d-2\delta)/\alpha}\vf(z)\vf\big((e^t-1)^{1/\alpha}z\big)h^2(z) \ud z \\
&=\int_\Rdz \vf((e^{-t})^{1/\alpha}x)\vf\big((1-e^{-t})^{1/\alpha}x\big)h^2(x) \ud x <\infty.\label{e.dvf0}
\end{align}
Since $-2d-2\alpha+2\delta<-d-3\alpha<-d$, the finiteness in \eqref{e.dvf0} follows from Corollary \ref{c.bphi}.
This proves  the continuity of the extension of $L_1\vf$ on the whole of $\Rd$. The rest of the lemma follows from the continuity, Lemma~\ref{l.crho} and \eqref{risom}, and from Corollary~\ref{c.bphi}.
\end{proof}
Needless to say, the continuous modification of $\vf$ is unique and pointwise defined for every $x\in \Rd$. From now on the extension shall be denoted by $\vf$. 
By the equality in \eqref{e.dvf0}, 
\begin{equation}\label{e.wvf0}
\vf(0)=\int_\Rdz \vf\big(\lambda^{1/\alpha}x\big)\vf\big((1-\lambda)^{1/\alpha}x\big)h^2(x) \ud x 
\end{equation}
for every $\lambda \in [0,1]$, including the endpoint cases, since $\varphi$ is a density.
In particular,
\begin{equation}\label{e.wvf02}
\vf(0)=\int_\Rd \vf(2^{-1/\alpha}x)^2 h^2(x) \ud x .
\end{equation}

\subsection{Regularization of $\rho$}

We are now in a position to prove
convergence of $\rho_1(x,y)$ as $x\to 0$ to a finite limit. By scaling, Chapman-Kolmogorov, Lemma~\ref{l.crho} and the boundedness of $\vf$, for $y\in \Rdz$ and $\Rdz\ni x\to 0$ we get
\begin{eqnarray*}
\rho_1(x,y)&=&2^\frac{d-2\delta}{\alpha}\rho_2(2^{1/\alpha}x,2^{1/\alpha}y)\\
&=&2^\frac{d-2\delta}{\alpha}\int \rho_1(2^{1/\alpha}x,z)\rho_1(z,2^{1/\alpha}y)h^2(z) \ud z \\
&\to&
2^\frac{d-2\delta}{\alpha}\int \vf(z)\rho_1(z,2^{1/\alpha}y)h^2(z) \ud z \\
&=&\int \vf(z)\rho_{1/2}(2^{-1/\alpha}z,y)h^2(z) \ud z =L_{\ln 2}\vf(y)=\vf(y).
\end{eqnarray*}
Thus, for every $y\neq 0$,
\begin{equation}\label{rho_lim}
	\rho_1(0,y):=\lim_{x\to 0} \rho_1 (x,y) =\vf(y).
\end{equation}
By scaling, for all $t>0$ we get
\begin{equation}\label{e.srty}
	\rho_t(0,y) := \lim_{x\to 0} \rho_t (x,y)= t^{\frac{2\delta -d}{\alpha}}\rho_1\left(0, t^{-1/{\alpha}}y\right)=t^{\frac{2\delta-d}{\alpha}} \vf \left(t^{-1/\alpha}y \right).
\end{equation}
By \eqref{e:rhois1}, for $x\not=0$ we have
\begin{align*}
	\int_{\Rd} \rho_1(x,y)|y|^{-2\delta}  \ud y  = 
	1.
\end{align*}
By
\rf{eq:mainThmEstrho2} and \rf{rho_lim},
\begin{equation}\label{rhosim1}
	\rho_{1} (x,y)|y|^{-2\delta} \approx 
	(1+|y|)^{-d-\alpha+\delta}|y|^{-2\delta}, \quad |x|<1, \ y\neq 0.
\end{equation} 
Thus, applying \rf{rho_lim} and the dominated convergence theorem, we get
\begin{equation}\label{est_t}
	\int_{\Rd} \rho_t(0,y)h^2(y)  \ud y 	=1,\quad t>0.
\end{equation}
By the symmetry of $\rho_t$ and \eqref{e.srty}, for all $t>0$, $x\in \Rdz$ we also have
\begin{equation}\label{e.srtx}
	\rho_t(x,0) := \lim_{y\to 0} \rho_t (x,y)=\rho_t(0,x)= t^{\frac{2\delta -d}{\alpha}}\rho_1\left(t^{-1/{\alpha}}x,0\right)=t^{\frac{2\delta-d}{\alpha}}\vf \left( t^{-1/\alpha}x\right).
\end{equation}
\begin{lemma}\label{Lemma_rho_lim_0} 
	The function $\rho$ has a unique continuous positive extension to $(0,\infty)\times \Rd\times \Rd$.
\end{lemma}
\begin{proof}
	By Chapman-Kolmogorov and the symmetry of $\rho_1$, for $x,y\neq 0$, we have
	\begin{align}\label{rho1_2}
		\rho_1(x,y) = \int_{\Rd} \rho_{\nicefrac{1}{2}} (z,x) \rho_{\nicefrac{1}{2}}(z,y)  \ud   z .
	\end{align}
	Recall that $\rho_1$ is continuous on $\Rdz\times \Rdz$. This and \rf{e.srty} yield for all $z\in \Rdz$ 
that $\rho_{\nicefrac{1}{2}} (z,x) \to \rho_{\nicefrac{1}{2}} (z,x_0)$  and $\rho_{\nicefrac{1}{2}} (z,y)\to \rho_{\nicefrac{1}{2}} (z,y_0)$ if $x\to x_0\in \Rd$ and $y\to y_0\in \Rd$, hence 
	\begin{equation}\label{d.r100}
		\rho_1 (x,y) \to \int_{\Rd}\rho_{\nicefrac{1}{2}} (z,x_0)\rho_{\nicefrac{1}{2}} (z,y_0)  \ud   z , 
	\end{equation}
	by the dominated convergence theorem, since
	for bounded $x,y\neq 0$, by
	and \rf{e.crt1t2} and \rf{eq:mainThmEstrho2},
	\begin{equation}\label{rhosim}
		\rho_{\nicefrac{1}{2}} (z,x)\rho_{\nicefrac{1}{2}} (z,y) \approx 
		(1+|z|)^{-2d-2\alpha+2\delta}.
	\end{equation} 
	The latter function
	is integrable because $-2d-2\alpha+2\delta<-d-3\alpha<-d$.
	Furthermore,
	\eqref{rhosim} implies that the limit in \eqref{d.r100} is positive and, by \rf{rho1_2}, it is an extension of $\rho_1$ to $\Rd\times \Rd$, which we shall denote by $\rho_1$ again.
	In view of \rf{ss form rho},
	\begin{equation}\label{e.et}
t^{\frac{2\delta-d}{\alpha} } \rho_1\big(t^{-1/\alpha }x , t^{-1/\alpha }y \big), \quad x,y\in \Rd,\quad t>0,
	\end{equation}
	is finite and defines a continuous extension of $\rho_t$ for each $t>0$. The 
extension is clearly positive and jointly continuous in $t,x,y$.
It is also unique, since $(0,\infty)\times \Rdz\times \Rdz$ is dense in $(0,\infty)\times \Rd\times \Rd$.
\end{proof}
We recall that the boundedness of $\rho_1$ near the origin readily follows from \eqref{eq:mainThmEstrho}. On the other hand the continuity 	turned out difficult for us to capture directly or indirectly. For instance we  considered this connection in \cite{MR4164845}. The proof presented above seems to be underpinned by self-regularization of functions satisfying (implicit) integral equations.

In what follows, $\rho$ will denote the continuous extension of $\rho$ to $(0,\infty)\times \Rd\times \Rd$. 
\begin{corollary}
	$\rho_1(0,0)= \lim_{x,y\to 0} \rho_1(x,y)\in (0,\infty)$.
\end{corollary}
The following Chapman-Kolmogorov equations hold, so $\rho$ is a transition density on $\Rd$.
\begin{corollary}\label{c.CH} For all $s,t>0$, $x,y\in \Rd$, we have
	$
	\int_\Rd \rho_s(x,z)\rho_t(z,y) h^2(z) \ud z =\rho_{s+t}(x,y).
	$
\end{corollary}
\begin{proof}
	Since $(1+|z|)^{-2d-2\alpha+2\delta}|z|^{-2\delta}$ is integrable, we can use \rf{rho_lim} and \rf{e.CKfrho}.
\end{proof}
\begin{proof}[Proof of Theorem~\ref{t:eta}]
The first statement of the theorem is proved in Lemma~\ref{Lemma_rho_lim_0}, \eqref{ss form eta} is proved more generally as \eqref{e.et}, and \eqref{e:ssf} is proved more generally in
Corollary~\ref{c.CH}. See also \eqref{rho_lim} and \eqref{e.srty} for a more explicit expression of \eqref{e.det}.
\end{proof}

\begin{proof}[Proof of Theorem~\ref{t:eta2}]
The result is immediate from \eqref{e:drhot} and \eqref{e.det}.
\end{proof}

We can now extend \eqref{eq:mainThmEstrhot} as follows.
\begin{corollary}\label{cor:mainThmEstrhot}
\begin{align*}
\rho_t(x,y)\approx \left( t^{-d/\alpha}\wedge \frac{t}{|x-y|^{d+\alpha}} \right)\left(t^{\delta/\alpha}+|x|^{\delta}  \right)\left(t^{\delta/\alpha}+|y|^{\delta}  \right) ,\quad 
t>0,  \ x,y \in \Rd.
\end{align*}
\end{corollary}

\begin{example}
For $\kappa=0$ we have $H=h=1$, $\rho=\tilde p=p$, $l_t(x,y)=p_{1-e^{-t}}(e^{-t/\alpha}x,y)$, $\varphi(x)=p_1(0,x)$, and $\Psi_t(x)=p_t(0,x)$.
\end{example}

\section{Functional analysis of $\tilde P_t$}\label{s.fatp}
We recall that
$\tilde P_t\vf (x) := \int_{\R^d} \tilde p(t,x,y) \vf (y) \ud y$.
By \eqref{eq:2}, $\tilde P_t h = h$ for all $t>0$.

\begin{lemma}\label{contraction}
$\{\tilde P_t\}_{t>0}$
 is a contraction semigroup on $L^1(h)$ and for every $f\in L^1(h)$, we have
\begin{equation}\label{e.cl}
\int_\Rd \tilde P_t f(x) h(x)\ud x=\int_\Rd f(x)h(x)\ud x,\quad t>0.
\end{equation}
\end{lemma}
\begin{proof}
The semigroup property of $\tilde P_t$ follows from \eqref{eq:ChK}.
Let $ f   \ge  0$ be a measurable function. By Fubini-Tonelli, the symmetry of $\tilde p$ and \eqref{eq:2}, 
\begin{align*}
 \int_\Rd \tilde P_t  f  (x)  h(x) \ud x & = \int_\Rd \int_\Rd \tilde p(t, x,y)  f  (y) h(x) \ud y \ud x
    = \int_\Rd h(y)  f  (y) \ud y.
\end{align*}
For arbitrary $ f \in L^1(h)$ we write $ f = f _+- f _-$ and use the nonnegative case. 
\end{proof}

\begin{lemma}\label{contr.L1}
$\{P_t\}_{t>0}$ is a strongly continuous contraction semigroup on $L^1(h)$. 
\end{lemma}
\begin{proof}
Since $p\le\tp$, for nonnegative function $ f $ and $t>0$ we get
$\|P_t  f \|_{L^1(h)}\le \|\tP_t  f \|_{L^1(h)}\le \| f \|_{L^1(h)}$, so the contractivity follows as in the proof of Lemma~\ref{contraction}.
 
By the contractivity and the semigroup property of $P_t$,  it suffices to prove strong continuity as $t\to 0$. Let $ f  \in L^1(h)$. Then $g:= f  h\in L^1$. There are functions $g_n \in C_c^{\infty } \left( \Rdz \right)$ such that $\|g-g_n\|_{L^1} \rightarrow 0$ as $n\rightarrow \infty$. Let $ f _n=g_n/h$. Of course, $ f _n\in C_c^\infty(\Rdz)$ and $\| f - f _n\|_{L^1(h)}=\|g-g_n\|_{L^1}$ for every $n$, and, by the contractivity, we have
\begin{align*}
   \|P_t  f  -  f  \|_{L^1(h)} & \le  \|   P_t  f -P_t  f _n\|_{L^1(h)} + \| P_t  f _n -  f _n\|_{L^1(h)}+\|  f _n -  f \|_{L^1(h)}  \\
   &\le 2\|   f - f _n\|_{L^1(h)} + \| P_t  f _n -  f _n\|_{L^1(h)}.
\end{align*}
Therefore, it suffices to prove that $ \|P_t  f  -  f  \|_{L^1(h)}\to 0$ as $t\to 0$
for every $ f \in C_c^{\infty } \left( \Rdz \right)$. To this end take
$K\in (0,\infty)$ such that
$\text{supp }  f   \subset B(0, K/2)$. Then,
\begin{align*}
  \|P_t  f  &-  f  \|_{L^1(h)} \le \int_{|x|  \le  K} | P_t  f  (x) -  f  (x) | h(x) \ud x + \int_{|x| >K} P_t| f | (x)h(x) \ud x.
\end{align*}
When $t\to 0$, the above converges to $0$ because $P_t  f  \to  f $ uniformly, $h$ is locally integrable and $P_t| f |(x)\le ct(1+|x|)^{-d-\al}$ on $B(0,K)^c$, see \eqref{eq:oppt}. The proof is complete.

\end{proof}

\begin{lemma}\label{strong1}
$\{\tilde P_t\}_{t>0}$ is a strongly continuous contraction semigroup on $L^1(h)$.
\end{lemma}
\begin{proof} 
The contractivity is resolved in Lemma~\ref{contraction}.
Since $P_t$  is strongly continuous as $t\to 0$, by \eqref{eq:Df2} it suffices to consider nonnegative $ f \in L^1(h)$ and verify that

$$
  I_t := \int_{\bbfR^d} \int_{\bbfR^d} \int_0^t \int_{\bbfR^d} \tilde p(s, x,z)  q(z)p({t-s},z,y) f  (y)\ud z \ud s \ud y h(x) \ud x \to 0
$$
as $t\to 0$.
By \eqref{eq:2} and \cite[Lemma 3.3]{MR3933622},
\begin{align*}
I_t&=\int_{\bbfR^d} \int_0^t \int_{\bbfR^d} h(z)q(z)p({t-s}, z,y)  f  (y) \ud z \ud s \ud y  \\
&=\int_{\bbfR^d} \big(h(y)- P_t h(y)\big)  f  (y) \ud y =\|f\|_{L^1(h)}-\|P_t f \|_{L^1(h)} \to 0,
\end{align*}
indeed, because of Lemma~\ref{contr.L1}.
\end{proof}
We consider the resolvent operators $R_\lambda = \int_0^{\infty } e^{-\lambda t}P_t \ud t$ and $\tilde R_\lambda = \int_0^{\infty } e^{-\lambda t}\tP_t \ud t$ for $\lambda>0$. The contractivities of $P_t$  and $\tP_t$ yield the following result.
\begin{corollary}
$\lambda \tilde R_\lambda $ and $\lambda R_\lambda $ are contractions on $L^1(h)$.
\end{corollary}
Let 
${\mathcal L}$
 be the $L^1(h)$-generator of the semigroup $P_t$, with domain ${\mathcal D}$, and let 
 $\tilde{\mathcal L}$
 be the $L^1(h)$-generator of the semigroup $\tilde P_t$, with domain $\tilde {\mathcal D}$. 
The following result is consistent with \eqref{eq:SchrOp}.
\begin{proposition}
We have ${\mathcal D}\subset \tilde {\mathcal D}$ and 
$\tilde {\mathcal L}  f  = {\mathcal L} f  + q  f $
 for $ f   \in {\mathcal D}$. 
\end{proposition}
\begin{proof}
Let $\lambda>0$. Recall that $ f  \in {\mathcal D}$ if and only if $ f  = R_\lambda g$ for some (unique) $g\in L^1(h)$ and then $
{\mathcal L}\varphi 
= \lambda R_\lambda g-g$.
The same holds for $\tilde R_\lambda $. 
By integrating \eqref{eq:Df2} with respect to $e^{-\lambda t} \ud t$ and using Fubini-Tonelli and the following self-explanatory notation we get 
\begin{align*}
    \tilde R_\lambda (x,y) &= R_\lambda (x, y) + \int_0^{\infty } e^{-\lambda t } \int_0^t \int_{\R^d} \tilde p(s,x,z) q(z) p(t-s,z,y) \ud z \ud s\ud t \\
    & = R_\lambda (x, y) + \int_0^{\infty }  \int_s^{\infty }\int_{\R^d}  e^{-\lambda s }\tilde p(s,x,z) q(z) e^{-\lambda (t -s)}p(t-s,z,y) \ud z \ud t\ud s \\
   & = R_\lambda (x, y) + \int_{\R^d}\tilde R_{\lambda } (x, z) q(z) R_{\lambda } (z, y) \ud z,\quad x,y \in \Rd.
\end{align*}
Accordingly, for $ f   \ge  0$, we get 
$ \tilde R_{\lambda }  f  = R_{\lambda }  f  + \tilde R_{\lambda } q R_{\lambda }  f $.
In particular, $\tilde R_{\lambda } q R_{\lambda }  f  \in L^1(h)$ if $ f  \in L^1(h)$. 
Moreover, $q R_{\lambda }$ is a bounded operator on $L^1(h)$.
Indeed, if $0\le  f  \in L^1(h)$, then 
using Fubini-Tonelli,  the identity $e^{-\lambda t}=\int_t^\infty \lambda e^{-\lambda s}ds$ and \cite[Lemma 3.3]{MR3933622}, we get 
\begin{align*}
0&\le \int_\Rd q(x)R_\lambda  f  (x)h(x) \ud x = \int_\Rd \int_{\Rd}\int_0^\infty e^{-\lambda t} p(t,x,y) q(x)h(x)   f  (y) \ud t\ud y \ud x \\
&= \int_0^\infty \int_\Rd \lambda e^{-\lambda s}  f  (y) \left(\int_0^s\int_{\Rd} p(t,x,y) q(x)h(x)  \ud x \ud t \right) \ud y \ud s \\
&\le \int_0^\infty \lambda e^{-\lambda s}\ud s \int_\Rd h(y)  f  (y) \ud y   = \int_\Rd h(y)  f  (y) \ud y.
\end{align*}
It follows that $R_{\lambda } = \tilde R_{\lambda } (I  - qR_{\lambda })$ on $L^1(h)$, and so
\begin{equation*}
   {\mathcal D} = R_{\lambda } \big( L^1(h) \big) \subset \tilde R_{\lambda } \big( L^1(h)\big) = \tilde{\mathcal D}.
\end{equation*}
To prove that 
$\tilde {\mathcal L} \varphi ={\mathcal L}\varphi +q$ 
for $\varphi \in \mathcal{D}$, we let 
$\psi ={\mathcal L}\varphi$, 
$\varphi, \psi \in L^1(h)$, that is $\varphi =\lambda R_\lambda  f  -R_\lambda \psi$. It is enough to prove that 
$\tilde {\mathcal L} f  = \psi +q f $, 
that is 
\begin{equation}\label{vf_psi}
	 f  = \lambda \tilde R_\lambda  f  -\tilde R_\lambda q f  .
	\end{equation}
But the right-hand-side of \rf{vf_psi} is
\begin{align*}
	\lambda R_\lambda  f  +\tilde R_\lambda \lambda q R_\lambda  f  -R_\lambda \psi -R_\lambda qR_\lambda \psi -\tilde R_\lambda q f  = f  +\tilde R_\lambda q  f  -\tilde R_\lambda q  f  = f .
	\end{align*}
The proof is complete. 
\end{proof}
Recall the notation of \rf{normqh}.
Our proof of the large-time asymptotics of $\tilde P_t$ hinges on the following hypercontractivity result (for $q=1$ see Lemma \ref{contraction} and for $q=\infty$ see Lemma~\ref{l.qinf}).
\begin{theorem}\label{hyperc}
If $1<q< \infty $	then for all $t>0$ and nonnegative functions $f$ on $\Rd$,
	\begin{equation}\label{w norm}
		\|\tilde P_t f\|_{q,h} \le  Ct^{-\frac{d-2\delta}{\alpha}(1-\frac{1}{q})} \|f\|_{L^1(h)}+Ct^{-\frac{d-2\delta}{\alpha}(1-\frac{1}{q})-\frac{\delta}{\alpha}}\|f\|_{L^1(\R^d)}.
	\end{equation}
\end{theorem}
\begin{proof}
We first prove the inequality \rf{w norm} for $t=1$. Let $q'=\frac{q}{q-1}$.  It is enough to verify that for every function $g\ge 0$ on $\R^d$,
	\begin{align}\label{observation}
		\int_{\R^d} \tilde P_1 f(x)g(x)
		h^{2-q}(x) 
		\ud x \le C \|f\|_{L^1(H)}
		 \|g\|_{L^{q'}(h^{2-q})},
	\end{align} 
	since then, by {the duality of $L^p(h^{2-q})$}, we get
	\begin{equation}\label{e.op1}
\|\tilde P_1 f\|_{q,h}
=\|\tilde P_1 f\|_{L^q(h^{2-q})}=\sup_{ \|b\|_{L^{q'}(h^{2-q})}=1}\left|\int_{\R^d} \tilde P_1 f(x)b(x)h^{2-q}(x)\ud x\right|
 \le  C\|f\|_{L^1(H)},
	\end{equation}	
as needed. To prove \rf{observation}, by \eqref{eq:mainThmEstz} 
it is enough to estimate
	\begin{align}
		I(f,g) &:= \int_{\R^d}\int_{\R^d} H(x)H(y)p(1,x,y)f(y)g(x)h^{2-q}(x) \ud   y   \ud x  \nonumber\\
		&=I(f\indyk_B,g\indyk_{B})+I(f\indyk_B,g\indyk_{B^c})+I(f\indyk_{B^c},g\indyk_{B})+I(f\indyk_{B^c},g\indyk_{B^c}),
		\label{e.4t}
	\end{align}
	where $B=B(0,1)\subset \R^d$. By H\"older's inequality,
	\begin{align*}
	I(f\indyk_B,g\indyk_{B}) & \le  \int_B \int_B h(x)h(y)p(1,x,y)f(y)g(x) \ud y  h(x)^{2-q} \ud x  \\
	& \le c \left( \int_B h(y)f(y) \ud y \right) \left( \int_B h(x)g(x) h^{2-q}(x)\ud x  \right)\\
	& \le c \|f\|_{L^1(H)} \left(\int_B g^{q'}(x) h^{2-q}(x)  \ud x \right)^{\frac{1}{q'}}\left( \int_Bh^{2}(x)  \ud x \right)^{\frac{1}{q}}\\
	& \le  C\|f\|_{L^1(H)} \|g\|_{L^{q'}(h^{2-q})},
\end{align*}
where $\int_Bh^{2}(x) \ud x<\infty$, because $2\delta \le  d-\alpha<d$. We next deal with the second term in \eqref{e.4t},
\begin{align} \label{last}
	\nonumber
	I(f\indyk_B,g\indyk_{B^c}) & \le  \int_B \int_{B^c} f(y)h(y)p(1,x,y)g(x) h^{2-q}(x)\ud x  \ud y 
\\ 	\nonumber& \le  c\int_B \int_{B^c} f(y)h(y)|x|^{-d-\alpha}g(x) h^{2-q}(x)\ud x   \ud y  \\
	\nonumber
	& \le c \|f\|_{L^1(H)}\int_{B^c} |x|^{-d-\alpha} g(x) h^{2-q}(x)\ud x   \\
	&\le  c\|f\|_{L^1(H)}\left( \int_{B^c} g^{q'}(x)h^{2-q}(x) \ud x \right)^{\frac{1}{q'}}\left(\int_{B^c}|x|^{-(d+\alpha)q}h^{2-q}(x) \ud x  \right)^{\frac{1}{q}}\\
	\nonumber& \le   C\|f\|_{L^1(H)} \|g\|_{L^{q'}(h^{2-q})}.
\end{align}
	The last integral in \rf{last} is finite because $0\le \delta \le (d-\alpha)/2$, so  $|x|^{-(d+\alpha)q-\delta(2-q)} \le  |x|^{-d-\alpha} \in L^1(B^c,  \ud x )$ for every $q\in (1,+\infty)$. We then estimate the third term,
	\begin{align*}
		I(f\indyk_{B^c},g\indyk_{B})& \le  \int_{B^c}\int_B f(y)h(y)p(1,x,y)g(x) h^{2-q}(x) \ud x   \ud y \\
		&  \le c \int_{B^c} \int_B f(y)h(y)|y|^{-d-\alpha} g(x) h^{2-q}(x)\ud x   \ud y  \\
		& \le c \|f\|_{L^1(H)}\int_B g(x)h^{2-q}(x) \ud x  \\
		& \le  c\|f\|_{L^1(H)}\left(\int_B g^{q'}(x) h^{2-q}(x) \ud x \right)^{\frac{1}{q'}}\left( \int_B h^{2-q}(x)  \ud x  \right)^{\frac{1}{q}}, 
	\end{align*}
	and the last integral is finite, as before.
For the fourth term in \eqref{e.4t} we note that
	\begin{align*}
		\int_{\R^d} p(1,x,y)g(x)h^{2-q}(x)\ud x & \le  \left( \int_{\R^d} g^{q'}(x)h^{2-q}(x) \ud x \right)^{\frac{1}{q'}} \left( \int_{\R^d} p^q(1,x,y) h^{2-q}(x) \ud x  \right)^{\frac{1}{q}}\\
		& \le  \|g\|_{L^{q'}(h^{2-q})}\left( \int_{\R^d} p_1^q (x) h^{2-q}(x)  \ud x \right)^{\frac{1}{q}},
	\end{align*}
	by rearrangement inequalities \cite[Sec.3]{MR1817225}. The last integral is finite, because
	\begin{align*}
		p_1^q(x)h^{2-q}(x) & \le  C|x|^{-\delta},\quad x\in B,\\
		p_1^q(x)h^{2-q}(x) & \le  C|x|^{-(d+\alpha)q -\delta(2-q)}  \le  |x|^{-d-\alpha}, \quad x\in B^c.
	\end{align*}
Thus,
	\begin{align*}
		I(f\indyk_{B^c},g\indyk_{B^c})& \le  \int_{B^c}\int_{B^c} p(1,x,y)f(y)g(x)h^{2-q}(x)  \ud x   \ud y \\
		&  \le  C \|g\|_{L^{q'}(h^{2-q})}\int_{B^c} f(y)\ud y  \le  C\|f\|_{L^1(H)} \|g\|_{L^{q'}(h^{2-q})}.
	\end{align*}
Therefore \eqref{e.4t}	is bounded above by $C\|f\|_{L^1(H)} \|g\|_{L^{q'}(h^{2-q})}$, which yields \eqref{w norm} for $t=1$. 

 By this and a change of variables, for all $t>0$ we get
	\begin{align}
		\|\tilde P_t f\|_{q,h}&=
		t^{\frac{d-\delta(2-q)}{\alpha q}}\| \tilde P_1 f(t^{\nicefrac{1}{\alpha}}\cdot )\|_{q,h} \label{e.spt}
		\\ \nonumber
		& \le  Ct^{\frac{d-\delta(2-q)}{\alpha q}}\left( \| f (t^{\nicefrac{1}{\alpha}}\cdot) \|_{L^1(h)}+\| f(t^{\nicefrac{1}{\alpha}}\cdot )\|_{L^1}\right)\\ \nonumber
		&=Ct^{\frac{d-\delta(2-q)}{\alpha q}}\left( t^{-\frac{d-\delta}{\alpha}}\|f\|_{L^1(h)}+t^{-\frac{d}{\alpha}}\|f\|_{L^1}\right)\\&=Ct^{-\frac{d-2\delta}{\alpha }(1-\frac{1}{q})}\|f\|_{L^1(h)}+Ct^{-\frac{d-2\delta}{\alpha }(1-\frac{1}{q})-\frac{\delta}{\alpha }}\|f\|_{L^1}.\nonumber
	\end{align}	
The proof of \eqref{w norm} is complete. 
\end{proof}
We note in passing that \eqref{w norm} for $t=1$ is equivalent to \rf{e.op1}
but $H$ is incompatible with the scaling \eqref{e.spt}, so
the long form of \eqref{w norm} seems inevitable.
Here is a more trivial bound.
\begin{lemma}\label{l.qinf}
$\|\tilde P_t f\|_{\infty,h}\le \|f\|_{\infty,h}$.
\end{lemma}
\begin{proof}
For $x\in \Rdz$ we have $|\tilde P_t f (x)\ud x|\le \|f\|_{\infty,h}\int_\Rd \tilde P_t(x,y)h(y)\ud y
\le h(x)\|f\|_{\infty,h} $.
\end{proof}

\section{Asymptotic behavior for large time}\label{s.ssa}
This section is devoted to the proof of Theorem \ref{lim norm th}. For the sake of 
comparison let us 
discuss the following classical analogue of 
\rf{equation}, 
\begin{equation}\label{lam p}
	\begin{cases}
		\partial_t u(x,t)=\left(\Delta +\kappa  |x|^{-2}\right) u(x,t) , \quad x\in \bbfR^d,\,t>0 ,\\
		u(x,0)=f(x),
	\end{cases}
\end{equation}  
where $d \ge  3$ and $\kappa  \in \bbfR$. As we already mentioned 
in Section \ref{s:Mm}, the Cauchy problem \rf{lam p} was popularized by Baras and Goldstein \cite{MR742415}, who discovered  that \rf{lam p} has no positive local-in-time solutions  if 
$\kappa  >(d-2)^2/4$, which  is called the instantaneous blow up. See also Goldstein and Kombe \cite{GK} for a simple proof via Harnack inequality. 
V\'azquez and Zuazua \cite{MR1760280} then studied the large time behavior of solutions to \rf{lam p} with $0<\kappa   \le  (d-2)^2/4$. Using a weighted version of the Hardy-Poincar\'e inequality, they proved in \cite[Theorem 10.3]{MR1760280} the stabilization of some solutions toward the following self-similar solution of \rf{lam p}
\begin{equation*}
	V(x,t)=t^{\sigma -\frac{d}{2}}|x|^{-\sigma }\se^{-\frac{|x|^2}{4t}},
\end{equation*}
where $\sigma=\frac{d-2}{2}-\sqrt{(d-2)^2/4- \kappa}$. Then, Pilarczyk \cite{MR3020137} proved that if $u=u(t,x)$ is the solution of \rf{lam p} with $\kappa  \in \big(-\infty , (d-2)^2/4\big)$, then for every $1 \le  q \le  \infty$,
\begin{equation*}
	\lim_{t\rightarrow \infty }t^{\frac{d}{2}(1-\frac{1}{q})-\frac{\sigma }{2}}\| u(\cdot ,t)-AV(\cdot , t)\|_{q, \vf_\sigma (t)}=0.
\end{equation*}
Here $f \in L^1(\bbfR^d) \cap L^1(|x|^{-\sigma}\ud x)$,
\begin{equation*}
	A=\frac{\int_{\bbfR^d} |x|^{-\sigma }
	f(x)   \ud x }{\int_{\bbfR^d} |x|^{-2\sigma }\se^{-\frac{|x|^2}{4}} \ud x },
\end{equation*}   
 $
	\vf_\sigma (x,t)= 1 \vee ({\sqrt{t}}/{|x|}) ^\sigma$, and
\begin{align*}
	\|
	g\|_{q,\vf_\sigma (t)}=
	\begin{cases}
	\Big( \int_{\bbfR^d}|
	g(x)|^q \vf_{\sigma }^{2-q}(x,t)  \ud x \Big)^{1/q} & \textrm{for}\quad 1 \le  q<\infty,\\ 
{\rm ess}\sup_{x \in \bbfR^d}|
	g(x)|/\vf_{\sigma }(x,t)  &\text{for}\quad q=\infty.
	\end{cases}
\end{align*}

	In Theorem \ref{lim norm th} we extend the results of 
	V\'azquez and Zuazua	\cite{MR1760280} and Pilarczyk \cite{MR3020137}
	to the Cauchy problem \eqref{equation}. In fact we also propose a novel description of the asymptotics, using simpler, time-independent norms $\|f\|_{q,h }$ defined in \eqref{normqh}.

We now return to the setting of \eqref{equation} and \rf{eq:dk0}. The solutions to \rf{equation} will be defined as $u(t,x)= \tilde P_t f(x)$, $t>0$, $x\in \Rd$ for $f\in L^1(H)$. By Theorem~\ref{t:eta2},
\begin{align*}
\Psi_t(x):=\lim_{y \to 0} \frac{\tilde{p}(t,x,y)}{h(y)}=\rho_t(0,x)h(x)
= t^{\frac{2\delta-d}{\alpha}}\vf\big( t^{-1/\alpha}x\big)h(x),\quad t>0, \quad x\in \Rd. \nonumber
\end{align*}
In particular, $\Psi_t(0)=\infty$. Since, for $s>0$,
\begin{align*}
\rho_{t+s}(0,x)&=\int_{\bbfR^d}\rho_t(0,y) \rho_s (y,x) h^2(y)\ud y= 
\int_{\bbfR^d}\rho_t(0,y) h(y) \tilde p (s,y,x)/h(x) \ud y,
\end{align*}
we get the following evolution property
\begin{align}\label{e.s}
\int_{\bbfR^d} \tilde p (s,y,x)\Psi_t(y) \ud y =\Psi_{t+s}(x). 
\end{align}
By \eqref{e.et} we also have
\begin{equation}\label{ss form eta1}
\Psi_t(x)= t^{\frac{\delta-d}{\alpha} } \Psi_1(t^{-1/\alpha }x), \quad t>0,\quad x\in \Rd.
\end{equation}
We summarize \eqref{e.s} and \eqref{ss form eta1} by saying that $\Psi_t(x)$ is a self-similar semigroup solution of \rf{equation}. By \rf{est_t},
\begin{equation}\label{psi_int}
	\int_{\Rd} \Psi _t (x) h(x)\ud x =\int_{\Rd} \Psi_1 (x) h(x) \ud x =1.
	\end{equation}
Furthermore, $\Psi_t(x)$ is a mild solution of \eqref{equation} since it satisfies the following Duhamel formula.
\begin{lemma}\label{lem:PsiDuh}
\begin{align}\label{e.DfPsi}
\Psi_t(x) &= \int_0^t\int_{\Rd} \Psi_r(z) q(z) p(t-r,x,z) \ud{z}\ud{r}, \quad t>0,\; x\in \Rd.
\end{align}
\end{lemma}
\begin{proof}
By \eqref{eq:mainThmEstz}, we have
\begin{align*}
\frac{\tp(t,x,y)}{h(y)} \approx (|y|^{\delta}+t^{\delta/\alpha})(1+t^{\delta/\alpha}|x|^{-\delta})p(t,x,y),
\end{align*}
hence
\begin{align}\label{eq:u_est}
\Psi_t(x) \approx 
t^{\delta/\alpha}(1+t^{\delta/\alpha}|x|^{-\delta})
\big(t^{-d/\alpha} \land \tfrac{t}{|x|^{d+\alpha}}\big), \qquad t>0,\; x \in \Rd.
\end{align}
Now, by \eqref{e.s} and \eqref{eq:Df1} for $0<s<t$ we obtain
\begin{align}
\Psi_t(x) &= \int_{\RR^d} \tp(t-s,x,y) \Psi_s(y)\ud y \nonumber \\
&+ \int_0^{t-s} \int_{\RR^d} \int_{\RR^d} p(r,x,y) q(y) \tilde{p}(t-s-r,y,z)  \Psi_s(z)\ud z \ud y\ud r \nonumber\\
&= \int_{\RR^d} p(t-s,x,y) \Psi_s(y)\ud y + \int_0^{t-s} \int_{\RR^d} p(r,x,y) q(y) \Psi_{t-r}(,y)\ud y\ud r \nonumber\\	
&= \int_{\RR^d} p(t-s,x,y) \Psi_s(y)\ud y + \int_s^t \int_{\RR^d} p(t-r,x,y) q(y) \Psi_r(y)\ud y\ud r. \label{eq:PsiDuh2}
\end{align}
Using \eqref{eq:u_est} and \eqref{ss form eta1}, we get
\begin{align*}
0 \le \limsup_{s\to0} \int_{\RR^d} p(t-s,x,z) \Psi_s(z) \ud{z}&\le c \limsup_{s\to0} (t-s)^{-d/\alpha} \int_{\RR^d} \Psi_s(z) \ud{z} \\
&=  c \limsup_{s\to0} (t-s)^{-d/\alpha} \int_{\RR^d} s^{(\delta-d)/\alpha} \Psi_1(s^{-1/\alpha}z) \ud{z} \\
& = c t^{-d/\alpha} \int_{\RR^d}  \Psi_1(z) \ud{z}  \lim_{s\to0} s^{\delta/\alpha}= 0. 
\end{align*}
Therefore,	taking $s \to 0$ in \eqref{eq:PsiDuh2}, we obtain \rf{e.DfPsi}.
\end{proof}

\begin{remark} 
By the self-similarity \rf{ss form eta1},
the definition of $\| \cdot \|_{q, h }$, the change of variables $y=t^{-1/\alpha}x$, and \rf{eq:u_est}, we have for all $t>0$
	\begin{equation}\label{e.sq}
		0<\|\Psi_t\|_{q, h }=t^{-\frac{d-2\delta}{\alpha}(1-\frac{1}{q})}\|\Psi_1\|_{q,h }<\infty,
	\end{equation}
see \rf{e.spt}.
It follows that
	\begin{equation}\label{e.sqst}
		\|\Psi_{st}\|_{q, h }=t^{-\frac{d-2\delta}{\alpha}(1-\frac{1}{q})}\|\Psi_s\|_{q,h }
	\end{equation}
and
\begin{align}\label{Psi_norm}
	\| \Psi_t\|_{1,h} = \| \Psi_1\|_{1,h}=1,
	\end{align}
which is the same as \rf{psi_int}. Furthermore,
	\begin{equation*}
		t^{-\frac{d-2\delta}{\alpha}(1-\frac{1}{q})}\| u(t, \cdot)-A\Psi_t\|_{q,h }=\|
		t^{\frac{d-\delta }{\alpha}}u\left(t,t^{\frac{1}{\alpha}} \cdot \right) -A\Psi_1\|_{q,h },
	\end{equation*}
in analogy with the results of 
Pilarczyk \cite{MR3020137}
 and V\'azquez and Zuazua	\cite{MR1760280}.
\end{remark}
\begin{remark}
Similar arguments yield the following result, to be proved in a forthcoming paper,
		\begin{equation}\label{H_asymp}
		\lim_{t\rightarrow \infty }t^{\frac{d}{\alpha}\left(1-\frac{1}{q}\right)-\frac{\delta}{\alpha}}\| u(t,\cdot )-A\Psi_t(\cdot )\|_{q, H_t}=0, 
	\end{equation}
where $1 \le  q \le  \infty$, $H_t(z) =H(t^{-1/\alpha}z)$ and the 
norms are defined by
\begin{equation*}
	\|f\|_{q,H_t}=
		\begin{cases}
	\bigg( \int_{\bbfR^d} |f(x)|^q H_t ^{2-q} (x)  \ud x  \bigg) ^{\frac{1}{q}}  & \textrm{for}\quad 1 \le  q<\infty,\\ 
\sup_{x\in \bbfR^d} |f(x)|/H_t(x)  &\text{for}\quad q=\infty.
	\end{cases}
\end{equation*}
The asymptotics \rf{H_asymp} is an exact analogue of the result in \cite[Theorem 2.1]{MR3020137}.
	\end{remark}
\begin{remark}\label{R:new1}
By  \eqref{e.cl} and \eqref{psi_int}, $A$ in Theorem~\ref{lim norm th} satisfies $\int_{\bbfR^d} \big(\tilde P_t  f(x)-A \Psi_s(x)\big)h(x)  \ud x =0$ for all $t\ge0, s >0$.
	\end{remark}
\begin{remark}\label{r.cit} By Lemma~\ref{contraction} we have $\int_\Rd u(t,x)h(x) \ud x=\int_\Rd f(x)h(x)\ud x$ for all $t\ge 0$. Here is a rather heuristic alternative argument for the equality:
	If we multiply \rf{equation} by the function $h$ and integrate, then indeed we get
	\begin{equation*}
		\frac{\ d }{\ d t} \int_{\bbfR^d} |x|^{-\delta }u(t,x)  \ud x  = -\int_{\bbfR^d}(-\Delta )^{\nicefrac{\alpha }{2}}  u(t,x) |x|^{-\delta }  \ud x  + \kappa_\delta \int_{\bbfR^d} |x|^{-\delta -\alpha }u(t,x)  \ud x .
	\end{equation*}
Using the Fourier symbol of $\Delta^{\alpha/2}$, Parseval's relation and \rf{eq:dk0} we get
	\begin{align*}
		\frac{\ d }{\ d t} \int_{\R^d} |x|^{-\delta }u(t,x)  \ud x  &= -\frac{2^{\alpha }\pi^{\alpha+\delta - \nicefrac{d}{2}} \Gamma \left( \frac{d-\delta}{2} \right)}{ \Gamma \left(\frac{\delta}{2}\right)}\int_{\R^d} |\xi |^{\alpha + \delta -d}\hat u(t,\xi ) \ud \xi + \kappa_\delta \int_{\R^d} |x|^{-\delta -\alpha }u(t,x)  \ud x  \\
		& =\Big( \kappa_\delta  - \kappa_\delta\Big) \int_{\bbfR^d} |x|^{ -\delta -\alpha } u(t,x)  \ud x  = 0,
	\end{align*}
see \cite[Theorem 2.2.14 p.102]{MR2445437} and \cite[Lemma 2 p.117]{MR0290095}.
	\end{remark}

Recall that $\tilde P_t$ is defined in \eqref{e.dtp}. The next lemma is crucial for the proof of Theorem~\ref{lim norm th}. 
\begin{lemma}\label{lim 0}
If $f\in L^1(H)$, $\int_{\bbfR^d } f(x)h(x) \ud x  =0$ and $1 \le  q <  \infty$, then 
\begin{equation}\label{lim_0_cond}
	\lim_{t \rightarrow \infty }t^{\frac{d-2\delta}{\alpha}(1-\frac{1}{q})}\| \tilde P_t f\|_{q,h }=0.
	\end{equation}
If, additionally, $f$ has compact support, then \eqref{lim_0_cond} is true for $q=\infty$, too.
\end{lemma}

\begin{proof}
	First, we consider a compactly supported function $\psi $ such that $\psi \in L^1(H)$ and
	\begin{equation}\label{0 con 1}
		\int_{\Rd} \psi (x) h(x)  \ud x =0 ,
	\end{equation}
and we intend to prove \rf{lim_0_cond} with $f$ replaced by $\psi$. 

\noindent
{\it Step 1. Case $q=\infty $.}\\
By the definition of the norm $\| \cdot \|_{\infty , h}$,
	\begin{align*}
		I(t)&:= t^{\frac{d-2\delta}{\alpha}}\| \tilde P_t \psi \|_{\infty ,h }
		=t^{\frac{d-2\delta}{\alpha}}\sup_{x\in \Rd} \left| \int_{\Rd} \rho _t(x,y)h(y)\psi(y) \ud y \right|.
	\end{align*}
	Using the condition \rf{0 con 1}, we obtain
	\begin{align*}
		I(t) = t^{\frac{d-2\delta}{\alpha}}\sup_{x\in \Rd}\left| \int_{\Rd} \left( \rho_t(x,y)-\rho_t(x,0)\right)h(y)\psi (y)  \ud y \right|.
	\end{align*}
	We fix $\omega>0$. Since $\psi$ has a compact support,  for sufficiently large $t>0$ we get
	\begin{align*}
		I(t) &=t^{\frac{d-2\delta}{\alpha}}\sup_{x\in \Rd}\left| \int_{|y| \le  t^{\frac{1}{\alpha}}\omega} \left( \rho_t(x,y)-\rho_t(x,0)\right)h(y)\psi (y)  \ud y \right|\\
		& \le   t^{\frac{d-2\delta}{\alpha}}\sup_{\substack {x\in \bbfR^d \\ |y| \le  \omega t^{\frac{1}{\alpha}}}}\big|\rho_t(x,y) -\rho_ t(x,0)\big| \int_{|y| \le  \omega t^{\frac{1}{\alpha}}}h(y)|\psi (y)|  \ud y ,
	\end{align*}
	so by scaling of $\rho$,
	\begin{align*}
		I(t)& \le   \sup_{\substack {x\in \Rd \\ |y| \le  \omega t^{\frac{1}{\alpha}}}}\left| \rho_1\left(t^{-\frac{1}{\alpha}}x,t^{-\frac{1}{\alpha}}y \right) - \rho_1\left( t^{-\frac{1}{\alpha}}x, 0\right) \right| \int_{|y| \le  t^{\frac{1}{\alpha}}\omega} h(y)| \psi(y)|  \ud y \\
		& \le  \sup_{\substack {x\in \Rd \\ |y| \le  \omega}}\left| \rho_1\left(x,y \right) - \rho_1\left( x, 0\right) \right| \int_{ \Rd}h(y)| \psi(y)|  \ud y.
	\end{align*}
By Lemma \ref{Lemma_rho_lim_0} and \rf{eq:mainThmEstrho2} we choose $\omega $ small enough to have
\begin{equation*}
	\sup_{\substack {x\in \Rd \\ |y| \le  \omega}}\left| \rho_1\left(x,y \right) - \rho_1\left( x, 0\right) \right|<\ve.
	\end{equation*}
 This proves \rf{lim_0_cond} in the considered setting.\\
 	{\it Step 2. Case $q=1$. }\\
By the definition of the norm $\| \cdot \|_{1, h}$,
	\begin{equation*}
		J(t):=\| \tilde P_t \psi \|_{1,h}= \int_{\Rd}\left| \int_{\Rd}\tilde p(t,x,y)h(x)\psi(y) \ud y \right| \ud x  =\int_{\bbfR^d }\Big| \int_{\bbfR^d }\rho_t (x,y)h^2(x)\psi (y)h(y)  \ud y  \Big|   \ud x .
	\end{equation*}
	Applying \rf{0 con 1}, we get
	\begin{equation*}
		J(t) \le  \int_{\bbfR^d } \int_{\bbfR^d } \big| \rho_t(x,y)-\rho_t(x,0)\big|h^2(x) |\psi (y)|h(y)  \ud y   \ud x .
	\end{equation*}
	We fix $\omega  >0$ and notice that 
	\begin{align*}
		J(t)& \le  \int_{\bbfR^d } \int_{|y| \le  \omega t^{\frac{1}{\alpha}}} \big| \rho_t (x,y)-\rho_t(x,0)\big|h^2(x) |\psi (y)| h(y)  \ud y   \ud x  ,
	\end{align*}
	for sufficiently large $t$, since the function $\psi $ has a compact support. Then
	\begin{align*}
		J(t)
		& \le  \sup_{|y| \le  \omega t^{\frac{1}{\alpha}}}\| \rho_t(\cdot , y)-\rho_t(\cdot ,0)\|_{L^1(h^2)} \int_{\bbfR^d} h(y)|\psi (y)|  \ud y .
	\end{align*}   
	By scaling of $\rho$, substituting $x=t^{\frac{1}{\alpha}}z$ we obtain
	\begin{align*}
		&\sup_{|y| \le  \omega t^{\frac{1}{\alpha}}}\| \rho_t(\cdot , y)-\rho_t(\cdot ,0)\|_{L^1(h^2)} \\
		&= t^{\frac{2\delta-d}{\alpha}} \sup_{|y| \le  \omega t^{\frac{1}{\alpha}}}\int \left | \rho_1 \left(t^{-{1}/{\alpha}}x,t^{-{1}/{\alpha}}y \right) -\rho_1\left( t^{-{1}/{\alpha}}x,0\right)\right|h^2(x) \ud x \\
		&=\sup_{|y|<\omega} \| \rho_1(\cdot ,y) -\rho_1(\cdot, 0)\|_{L^1(h^2)}.
	\end{align*}
	Applying Lemma~\ref{l.crho}, for $\omega >0$ sufficiently small we get
	\begin{equation*}
		J(t) \le  \ve \| \psi \|_{L^1(h)} .
	\end{equation*}   
	{\it Step 3. Case $q\in (1,\infty)$. }\\
	By the definition of the norm $\| \cdot \|_{q, h }$ and H\"older inequality,
	\begin{align*}
		t^{\frac{d-2\delta}{\alpha}(1-\frac{1}{q})}\|\tilde P_t \psi \|_{q, h }&=t^{\frac{d-2\delta}{\alpha}(1-\frac{1}{q})}\bigg(\int_{\R^d}\big| \tilde P_t\psi (x)/h(x)\big|^{q-1}\big|\tilde P_t\psi (x)h(x)\big|  \ud x \bigg)^{\frac{1}{q}} \\
		& \le  \left(t^{\frac{d-2\delta}{\alpha}}\|\tilde P_t\psi \|_{\infty ,h}\right)^{1-\frac{1}{q}} \left(|\tilde P_t\psi \|_{1, h} \right)^{\frac{1}{q}}.
	\end{align*}
	Both factors converge to zero as $t\rightarrow \infty $ by {\it Step 1.} and {\it Step 2.} \\
	{\it Step 4. General initial datum.}\\ 
	Let $R>0$, $c_R=\int_{|x|\le R} f(x)h(x)\ud x/ \int_{|x|\le R} h(x)\ud x$ and $\psi_R (x)= (f (x)-c_R)\indyk_{|x| \le  R}$. Of course, $\psi_R$ is
compactly supported and
$$\int_{\bbfR^d} h(x) \psi_R (x)  \ud x  =0.$$ 
Furthermore, 
	\begin{align*}
		\|f- \psi_R \|_{L^1(h )}& =   |c_R| \int_{|x| \le  R}h(x)  \ud x+\int_{|x|> R}h(x)|f (x)|  \ud x  \\
		& = \Big| \int_{|x| \le  R} h(x)f (x)  \ud x  \Big| + \int_{|x|> R}h(x)|f (x)|  \ud x \to 0 
	\end{align*}
 as $R\rightarrow \infty $, due to the assumption \rf{0 con 1} and the condition $f \in L^1(h)$. Let $\ve >0$ and choose $R>0$ so large that 
	\begin{equation*}
		\| f -\psi_R \|_{1,h}<\ve.
	\end{equation*}

For $q=1$, by using the triangle inequality and Lemma \ref{contraction}, we have
	\begin{align*}
		\|\tilde P_t f\|_{1,h }& \le  \|\tilde P_t\psi_R \|_{1,h }+ \|\tilde P_t(f-\psi_R) \|_{1,h }\\
		& \le  \|\tilde P_t\psi_R \|_{1,h }+ 
		\|f -\psi_R \|_{1, h }.
	\end{align*}
Hence, by {\it Step 2.} of this proof we get
	\begin{equation}\label{lim epsilon1}
		\limsup_{t\rightarrow \infty } \|\tilde P_t f\|_{1,h } \le  \ve,
	\end{equation}
which completes the verification of \rf{lim_0_cond} in this case.

If $1<q<\infty$, then using the triangle inequality and Theorem \ref{hyperc}, we obtain 
	\begin{align*}
		t^{\frac{d-2\delta}{\alpha}(1-\frac{1}{q})}\|\tilde P_t f\|_{q,h }& \le  t^{\frac{d-2\delta}{\alpha}(1-\frac{1}{q})}\|\tilde P_t\psi_R \|_{q,h }+ t^{\frac{d-2\delta}{\alpha}(1-\frac{1}{q})}\|\tilde P_t(f-\psi_R) \|_{q,h }\\
		& \le  t^{\frac{d-2\delta}{\alpha}(1-\frac{1}{q})}\|\tilde P_t\psi_R \|_{q,h }+ 
		C\|f -\psi_R \|_{1, h } +Ct^{-\frac{\delta}{\alpha}}\|f-\psi_R\|_{L^1(\R^d)}.
	\end{align*}
By {\it Step 3.} of this proof,
	\begin{equation}\label{lim epsilon}
		\limsup_{t\rightarrow \infty } t^{\frac{d-2\delta}{\alpha}(1-\frac{1}{q})}\|\tilde P_t f\|_{q,h } \le  
2C\ve,
	\end{equation}
which is valid both for $\delta>0$ and $\delta=0$.
This completes the proof of \rf{lim_0_cond} for $q\in (1,\infty)$.  
\end{proof}

We prove Theorem \ref{lim norm th} immediately after the following discussion of \textit{solutions} to \eqref{equation}.
\begin{remark}\label{r.ww}
{\rm
By \eqref{eq:PsiDuh2} we conclude that $\Psi_t(x)$ is a mild solution of  \eqref{equation}. 
 In passing we note that, by \eqref{eq:u_est}, $\Psi_t(x)\to 0$ when $t\to 0^+$ and $x\neq 0$.
Further, if $f\in L^1(H)$, $t>0$, $x\in \Rdz$ and $u(t,x)=\tilde P_t f(x)$, then 
\begin{align}
u(t,x) 
&= \int_{\RR^d} p(t,x,y) f(y)\ud y + \int_0^t \int_{\RR^d} p(t-r,x,y) q(y) u(r,y)\ud y\ud r, \label{eq:PsiDuh23}
\end{align}
as follows from \eqref{eq:Df1}.
This justifies calling $u$ in Theorem \ref{lim norm th}  solution to \eqref{equation}. 
By definition, $u$ can also be called the semigroup solution and \eqref{e.s} yields an analogue for $\Psi_t(x)$.
}
\end{remark}

\begin{proof}[Proof of Theorem \ref{lim norm th}]
By \rf{eq:ChK}, \rf{e.s}, Remark \ref{R:new1} and Lemma \ref{lim 0},
\begin{align*}
		\lim_{t\rightarrow \infty }t^{\frac{d-2\delta}{\alpha}(1-\frac{1}{q})}\| u(t, \cdot)-A\Psi_t\|_{q,h}& = \lim_{t\to \infty}t^{\frac{d-2\delta}{\alpha}(1-\frac{1}{q})}\|u(t+1,\cdot) -A\Psi_{t+1}\|_{q,h} \\
		&= \lim_{t\to \infty}t^{\frac{d-2\delta}{\alpha}(1-\frac{1}{q})}\|\tilde P_t\left( \tilde P_1f-A\Psi_1\right)\|_{q,h}=0.
	\end{align*}
\end{proof}

Theorem~\ref{lim norm th} is optimal, as asserted by the following two observations. 
\begin{proposition}\label{p.es}
Let $q\in[1,\infty)$ and $\tau:[0,\infty)\to [0,\infty)$ be increasing with $\lim_{t\to \infty}\tau(t)=\infty$. Then there is $f\in L^1(H)$ such that  $\int_{\Rd}f (x)  h(x) \ud x=1$ and $u:=\tilde P_t f$ satisfies
	\begin{equation}\label{e.ass}
		\lim_{t\rightarrow \infty }\tau(t)t^{\frac{d-2\delta}{\alpha}(1-\frac{1}{q})}\| u(t, \cdot)-\Psi_t\|_{q, h}=\infty.
	\end{equation}
\begin{proof}
The proof builds on the fact that $\|\Psi_{1}-\Psi_s\|_{q,h}>0$ for $s>1$, which, as we shall see below, is a consequence of the scaling of $\Psi_t(x)$.

Without loss of generality we may assume that $\tau$ is continuous, strictly increasing and $\tau(0)=0$, 
by replacing $\tau(t)$ by $t/(t+1)\left[t/(t+1)+\frac1t \int_0^t \tau(s)ds\right]\le \tau(t)+1$, see also \eqref{e.as}. In particular, $\tau^{-1}$ is well defined $: [0,\infty)\to [0,\infty)$.
For $n\in \bbfN$ we let $t_n=\tau^{-1}(2^{2n})$. 
Let
$f=\sum_{n=1}^\infty 2^{-n}\Psi_{t_n}$.
According to \eqref{Psi_norm}, 
$\int_{\Rd}f (x)  h(x) \ud x=\|f\|_{1,h}=1$. By \eqref{e.s},
$$u(t,x):=\tilde P_t f(x)=\sum_{n=1}^\infty 2^{-n} \Psi_{t+t_n}(x)\ge 0,\quad t>0,\ x\in \Rd.$$
We remark that if $\tau(t)\ge 1$, then $t_n\ge t$ is equivalent to 
$n\ge \left \lceil{\frac12 \log_2 \tau(t)}\right \rceil$. In this case,
by \eqref{e.sqst}, and triangle inequality, we get 
\begin{align}\label{e.sln}
&t^{\frac{d-2\delta}{\alpha}(1-\frac1q)}\| u(t, \cdot)-\Psi_t\|_{q, h}=
\|\sum_{n=1}^\infty 2^{-n} \Psi_{1+t_n/t}-\Psi_1\|_{q,h}\\
\nonumber
&\ge \|\Psi_1\|_{q,h} -\sum_{n=1}^\infty 2^{-n} \|\Psi_{1+t_n/t}\|_{q,h}
= \|\Psi_1\|_{q,h}\sum_{n=1}^\infty 2^{-n} \left(1-(1+t_n/t)^{-\frac{d-2\delta}{\alpha}(1-\frac{1}{q})}\right)\\
\nonumber
&\ge \|\Psi_1\|_{q,h}\left(1-2^{-\frac{d-2\delta}{\alpha}(1-\frac{1}{q})}\right)\sum_{t_n\ge t}^\infty 2^{-n}= \|\Psi_1\|_{q,h}\left(1-2^{-\frac{d-2\delta}{\alpha}(1-\frac{1}{q})}\right)
2^{1-\left \lceil{\frac12 \log_2 \tau(t)}\right \rceil}
\\
&\ge \|\Psi_1\|_{q,h}\left(1-2^{-\frac{d-2\delta}{\alpha}(1-\frac{1}{q})}\right)\tau(t)^{-1/2}.\nonumber
\end{align}
The case of $q\in (1,\infty)$ in \eqref{e.ass} is resolved, because  $1-2^{-\frac{d-2\delta}{\alpha}(1-\frac{1}{q})}>0$ in this case.

We next consider $q=1$. Then, starting from \eqref{e.sln},
we get 
\begin{align}
\| u(t, \cdot)-\Psi_t\|_{1, h}&=
\|\sum_{n=1}^\infty 2^{-n} \Psi_{1+t_n/t}-\Psi_1\|_{1,h}
\nonumber
\\\nonumber
&\ge \int_{B(0,1)}
|\Psi_1(x)-\sum_{n=1}^\infty 2^{-n} \Psi_{1+t_n/t}(x)|h(x)\ud x\\\nonumber
&\ge \int_{B(0,1)}\left(\Psi_1(x)-\sum_{n=1}^\infty 2^{-n} \Psi_{1+t_n/t}(x)\right)h(x)\ud x\\
&= \sum_{n=1}^\infty 2^{-n}\left(\int_{B(0,1)}\Psi_1(x)h(x)\ud x - \int_{B(0,1)}\Psi_{1+t_n/t}(x)h(x)\ud x\right).
\label{e.sln1aa}
\end{align}
By the same change of variables as in \eqref{e.sq}, for all $t>0$ and $U\subset \Rd$ we get
	\begin{align}\int_{U}\Psi_t(x) h(x)\ud x=
&\int_{t^{-1/\alpha}U}\Psi_1(x) h(x)\ud x.
\label{e.sq2a}
	\end{align}
Therefore by \eqref{e.sln1aa},
\begin{align*}
&\| u(t, \cdot)-\Psi_t\|_{1, h}
\ge \sum_{n=1}^\infty 2^{-n}\int_{B(0,1)\setminus B(0,(1+t_n/t)^{-1/\alpha})}\Psi_1(x)h(x)\ud x\\
&\ge \int_{B(0,1)\setminus B(0,2^{-1/\alpha})}\Psi_1(x)h(x)\ud x\sum_{t_n\ge t}^\infty 2^{-n}.
	\end{align*}
By \eqref{eq:u_est}, $\int_{B(0,1)\setminus B(0,2^{-1/\alpha})}\Psi_1(x)h(x)\ud x>0$, and we conclude as before.	
\end{proof}

\end{proposition}
\begin{remark}
We note that \eqref{e.as} does not hold for $q = \infty$. Indeed, let $f_n(x) = {\bf 1}_{B(x_n,1)}(x)$, where $\{x_n\}$ is a sequence in $\Rd$ such that $|x_n| = 2^{n}$. Then, by \eqref{eq:oppt}, for $t>1$ we have
\begin{align*}
\tP_t f_n(x_n) \ge \int_{B(x_n,1)} p(t,x_n,y) \ud{y} \ge c t^{-d/\alpha}.
\end{align*} 
Let $f(x) = \sum_{n=2}^{\infty} f_n(x)$. Clearly $f \in L^1(h)$ and $\tP_t f(x_n) \ge c t^{-d/\alpha}$. 
Moreover, by \eqref{eq:u_est} and \eqref{eq:oppt}, $\Psi_t(x_n) \le c  t^{\delta/\alpha} (1 + t^{-\delta/\alpha}|x_n|^{-\delta}) t|x_n|^{-d-\alpha}$. Hence for every $t>1$ we actually have,
\begin{align*}
\| u(t, \cdot)-A\Psi_t\|_{\infty, h} = \sup_{x \in \Rd} |\tP_t f(x) - A \Psi_t (x)|/h(x) \ge  c_t \sup_{n \in \NN}(|x_n|^{\delta} - |x_n|^{\delta-d-\alpha}) = \infty.
\end{align*}
\end{remark}

\section{The potential of the self-similar solution}\label{s.psss}

In this section we assume that $0\le \kappa \le \kappa^*$, except that on several occasions we explicitly exclude the critical case $\kappa=\kappa^*$.
As usual, $\kappa$ and $\delta$ are related by \rf{eq:dk0}. 

\textit{If} $\kappa < \kappa^*$ then, by \eqref{ss form eta1} and \eqref{eq:u_est}, there is $c\in (0,\infty)$, such that
\begin{align*}
c\int_0^\infty \Psi_s(x) \ud s = |x|^{\delta+\alpha-d}, \qquad x \in \Rd.
\end{align*}
Furthermore,
\begin{align*}
\int_{\R^d} \tilde{p}_t(x,y) |y|^{\delta+\alpha-d} \ud y &= c \int_0^\infty \int_{\R^d} \Psi_s(y) \tilde{p}_t(x,y) \ud y \ud s = c\int_0^\infty \Psi_{t+s}(x) \ud s\\ &= c\int_0^\infty \Psi_s(x) \ud s - c\int_0^t \Psi_s(x) \ud s = |x|^{\delta+\alpha-d} - c\int_0^t \Psi_s(x) \ud s.
\end{align*}
Hence, for $t>0$ and $y\in \RR^d \setminus \{0\}$,
\begin{align*}
c\int_0^t \Psi_s(x) \ud s = |x|^{\delta+\alpha-d}  - \int_{\R^d} \tilde{p}_t(x,y) |y|^{\delta+\alpha-d} \ud y.
\end{align*}
The main goal of this section is to calculate the constant $c$, and derive a similar formula for $\kappa=\kappa^*$ (see Corollary \ref{cor:intPsi}). 
\begin{lemma}\label{lem:PsiAI}
For $t>0$ and $x \in \Rd$ we have
\begin{align*}
\Psi_t(x) = \lim_{\beta \to 0^+} \frac{\Gamma(d/2)}{2\pi^{d/2}}\int_{\Rd} \beta |z|^{\beta+\delta-d} \tp(t,x,z) \ud{z}.
\end{align*}
\end{lemma}
\begin{proof}
Let $x \in \Rdz$. For $z\in\Rdz$ and $\beta>0$, put $f_\beta(z) = \frac{\Gamma(d/2)}{2\pi^{d/2}} \beta |z|^{\beta-d}$. Let $\eps >0$ and $\beta \to 0$. We have 
\begin{align*}
\beta \int_{B(0,\eps)} |z|^{\beta-d} \ud z &= d|B(0,1)| \beta \int_0^\eps  r^{\beta-d} r^{d-1} \ud{r} = d|B(0,1)|\eps^\beta  \longrightarrow \frac{2 \pi^{d/2}}{\Gamma(d/2)}.
\end{align*}
Furthermore, by \eqref{eq:mainThmEstz} and the dominated convergence theorem for every $x \in \Rdz$ we get
\begin{align*}
\lim_{\beta \to 0^+} \beta \int_{|z|>\eps} \frac{\tp(t,x,z)}{h(z)} |z|^{\beta-d} \ud{z} = 0.
\end{align*}
Then, since $z \mapsto \tp(t,x,z)/h(z)$ has a continuous extension to $\Rd$ with the value $\Psi_t(x)$ at the origin, we get
\begin{align*}
\Psi_t(x) = \lim_{\beta \to 0^+} \int_{\Rd} \frac{\tp(t,x,z)}{h(z)} f_\beta(z) \ud{z}.  
\end{align*}
The result follows (in the case of $x=0$, the statement is trivial).
\end{proof}
\begin{lemma}\label{lem:intlimPsi}
For $t>0$ and $x \in \Rdz$,
\begin{align*}
\int_0^t \Psi_s(x) \ud{s} = \lim_{\beta \to 0^+} \frac{\Gamma(d/2)}{2\pi^{d/2}} \frac{\beta}{\kappa_{\delta+\beta}-\kappa_\delta} \left( |x|^{-(d-\delta -\alpha-\beta)} - \int_{\Rd}  \tp(t,x,z)|z|^{-(d-\delta -\alpha-\beta)} \ud{z}\right).
\end{align*}
\end{lemma}
\begin{proof}
Fix $x \in \Rdz$ and $t>0$. By \eqref{eq:oppt}, there is a constant $c=c(d,\alpha)$ such that
\begin{align*}
&\int_{\Rd} p(s,x,z) |z|^{\gamma-d} \ud{z} \le \int_{|z|>|x|/2} p(s,x,z) (|x|/2)^{\gamma-d} \ud{z} + c\int_{|z|\le|x|/2} \frac{s}{|x|^{d+\alpha}} |z|^{\gamma-d} 
 \ud{z} \\
&\le (|x|/2)^{\gamma-d} + \frac{c}{\gamma} t|x|^{\gamma-d-\alpha} \le c \frac{t + 1}{\gamma} \left(1+ |x|^{-d-\alpha}\right), \quad s \in (0,t),\; \gamma \in (0,d).
\end{align*}
Then, by \eqref{eq:mainThmEstz2}, there is $c'=c'(d,\alpha,t)$ such that
\begin{align*}
\int_{\Rd} \beta \tp(s,x,z) |z|^{\beta+\delta-d} \ud{z} &\le c' H(x) \beta \int_{\Rd} p(s,x,z) (|z|^{\beta+\delta-d} + |z|^{\beta-d}) \ud{z} \\
&\le c'H(x)  \left(1 + |x|^{-d-\alpha} \right), \quad \beta \in (0,d-\delta). 
\end{align*}
Therefore, by Lemma \ref{lem:PsiAI}, the dominated convergence theorem and \cite[Theorem~3.1]{MR3933622},
\begin{align}
&\int_0^t \Psi_s(x) \ud{s} = \lim_{\beta \to 0^+}\frac{\Gamma(d/2)}{2\pi^{d/2}}\int_0^t  \int_{\Rd} \beta  \tp(s,x,z) |z|^{\beta+\delta-d} \ud{z} \ud{s} \notag\\
&= \lim_{\beta \to 0^+} \frac{\Gamma(d/2)}{2\pi^{d/2}} \frac{\beta}{\kappa_{d-\delta -\alpha-\beta}-\kappa_\delta} \left( |x|^{-(d-\delta -\alpha-\beta)} - \int_{\Rd}  \tp(t,x,z)|z|^{-(d-\delta -\alpha-\beta)} \ud{z}\right).
\label{eq:intu}
\end{align}
The result follows from the symmetry of $\kappa$.
\end{proof}

\begin{corollary}\label{cor:intPsi}
For $0 \le \delta< \frac{d-\alpha}{2}$, we have
\begin{align}\label{eq:intu2}
\int_0^t \Psi_s(x) \ud{s} = \frac{\Gamma(d/2)}{2\pi^{d/2}\kappa'_\delta} \left( |x|^{-(d-\delta -\alpha)} - \int_{\Rd}  \tp(t,x,z)|z|^{-(d-\delta -\alpha)} \ud{z}\right), \quad t>0,\, x\in\Rdz.
\end{align}
For $\delta= \frac{d-\alpha}{2}$, 
\begin{align}\label{e.pcc}
 \int_0^t \Psi_s(x)\ud{s} = \frac{\Gamma(d/2)}{ \pi^{d/2}\kappa_\delta''}  \left(|x|^{-\delta}\ln|x| - \int_{\R^d}  \tilde{p}(t,z,x) |z|^{-\delta}\ln|z| \ud{z}  \right), \quad t>0,\, x\in\Rdz.
\end{align}
Here, $\kappa_\delta'$ and $\kappa_\delta''$ are the first and the second derivatives of $\kappa_\delta$, respectively.
\end{corollary}
\begin{proof}
For  $0 \le \delta< \frac{d-\alpha}{2}$ the statement follows directly from Lemma \ref{lem:intlimPsi}, \eqref{eq:mainThmEstz} and the dominated convergence theorem. For $\delta= \frac{d-\alpha}{2}$, we have $\kappa_\delta' =0$ and
\begin{align}\label{eq:invfun}
\int_{\R^d}  \tilde{p}(t,x,z) |z|^{\delta+\alpha -d} \ud{z} = \int_{\R^d}  \tilde{p}(t,x,z) |z|^{-\delta} \ud{z} = |x|^{-\delta} = |x|^{\delta+\alpha-d},
\end{align}
see \rf{eq:2}, so the difference in the parentheses on the right hand side of \rf{eq:intu} tends to $0$.
So, as $\beta \to 0$, we use \eqref{eq:invfun}, \eqref{eq:mainThmEstz} and the dominated convergence theorem, obtaining
\begin{align*}
 \int_0^t \Psi_s(x)\ud{s} &= \lim_{\beta \to 0^+} \frac{\Gamma(d/2)}{2\pi^{d/2}} \frac{\beta^2}{\kappa_{\delta+\beta}-\kappa^*} \left(\frac{|x|^{\beta-\delta} - |x|^{-\delta} }{\beta}- \int_{\R^d}  \tilde{p}(t,x,z) \frac{|z|^{\beta-\delta} - |z|^{-\delta}}{\beta} \ud{z} \right) \\
&= \frac{\Gamma(d/2)}{ \pi^{d/2}\kappa_\delta''}  \left(|x|^{-\delta}\ln|x| - \int_{\R^d}  \tilde{p}(t,x,z) |z|^{-\delta}\ln|z| \ud{z}  \right).
\end{align*}
\end{proof}

\begin{theorem}\label{thm:intu}
For $\delta \in [0,\frac{d-\alpha}{2})$ and $x\in \Rd$ we have
\begin{align}\label{Eq:intu2}
 \int_0^\infty \Psi_s(x)\ud{s} = \frac{\Gamma(d/2)}{2 \pi^{d/2}\kappa_\delta'}  |x|^{\delta+\alpha-d},
\end{align}
and 
\begin{align}\label{Eq:intu}
 \int_t^\infty \Psi_s(x)\ud{s} = \frac{\Gamma(d/2)}{2 \pi^{d/2}\kappa_\delta'}   \int_{\R^d} \tilde{p}(t,x,z)|z|^{\delta+\alpha -d}  \ud z, \quad t>0.
\end{align}
For $\delta=(d-\alpha)/2$, 
\begin{align}\label{Eq:intuk}
 \int_t^\infty \Psi_s(x)\ud{s} = \infty, \quad x\in \Rd, \quad t\ge 0.
\end{align}
\end{theorem}
\begin{proof}
We shall prove \eqref{Eq:intu2} by letting $t \to \infty$ in \eqref{eq:intu2}. To this end let $x \in \Rdz$ and $T>1$. By \eqref{eq:intu2}, $\int_{\R^d}  \tilde{p}(t,x,z)|z|^{\delta+\alpha -d} \ud z $ is finite and decreases 
as $t$ increases to $\infty$. Hence, for every $\eps>0$ there is $R>0$ such that for every $t>T$,  
\begin{align}\label{eq:ptIntBr}
\int_{B(0,R)^c}  \tilde{p}(t,x,z) |z|^{\delta+\alpha -d} \ud z \le \int_{B(0,R)^c}  \tilde{p}(T,x,z) |z|^{\delta+\alpha -d} \ud z  < \eps. 
\end{align}
By  \eqref{eq:mainThmEstz}, for $t>T$ we get
\begin{align*}
|z|^{\delta+\alpha -d} \tilde{p}(t,x,z) &\le c |z|^{\delta+ \alpha -d} (1+t^{\delta/\alpha}|z|^{-\delta})(1+t^{\delta/\alpha}|x|^{-\delta}) t^{-d/\alpha} \\
&\le c |z|^{\delta+ \alpha -d} (1+|z|^{-\delta})(1+|x|^{-\delta}) T^{-(d-2\delta)/\alpha}.
\end{align*}
By the dominated convergence theorem,
\begin{align*}
\lim_{t\to\infty} \int_{B(0,R)}  \tilde{p}(t,x,z) |z|^{\delta+\alpha -d} \ud z  =0.
\end{align*}
This and \eqref{eq:ptIntBr} yield \eqref{Eq:intu2}.  Then \eqref{Eq:intu} follows by \eqref{eq:intu2} and \eqref{Eq:intu2}. If $x=0$, then we trivially have infinity on both sides of \eqref{Eq:intu2} and \eqref{Eq:intu}.
Finally, \eqref{Eq:intuk} follows from \eqref{eq:u_est}.
\end{proof}

\begin{remark}\label{r.asss}
We note that the function
\begin{align}
\mu_t(x) &:= \int_0^t \Psi_s(x)\ud s\label{e.dmutx}
\end{align}
is self-similar, too. Namely, by \eqref{ss form eta1} and changing variables $s=tu$ in \rf{e.dmutx}, we get
\begin{align*}
\mu_t(x)=t^{(\alpha+\delta-d)/\alpha}\mu_1(t^{-1/\alpha}x),\quad t>0, x\in \Rd.
\end{align*}
Furthermore, $\mu$ satisfies the Duhamel formula, 
\begin{align*}
\mu_t(x) = \int_0^t\int_{\Rd} \mu_s(z) q(z) p(t-s,x,z) \ud{z}\ud{s}.
\end{align*}
Indeed, by \eqref{e.DfPsi} and Fubini-Tonelli,
\begin{align*}
\int_0^t \int_{\RR^d} p(t-s,x,z)q(z)\mu_s(z) \ud{z} \ud{s} &= \int_0^t \int_{\RR^d} p(t-s,x,z)q(z)\int_0^s\Psi_{s-r}(z) \ud{r} \ud{z} \ud{s} \\
&= \int_0^t \int_r^t\int_{\RR^d} p(t-s,x,z)q(z)\Psi_{s-r}(z) \ud{s} \ud{z} \ud{r} \\
&= \int_0^t \int_0^{t-r}\int_{\RR^d} p(t-r-s,x,z)q(z)\Psi_s(z) \ud{s} \ud{z} \ud{r} \\
&= \int_0^t \Psi_{t-r}(z)  \ud{r} = \mu_t(x). \\
\end{align*}
Of course, it is $\Psi_t$, not $\mu_t$, that captures the large time asymptotics for the solutions of the equation \eqref{equation} in Theorem~\ref{lim norm th}. Interestingly, it seems feasible, if not easy, to construct $\mu_t$ directly as $\lim_{y\to0} \int_0^t \tp(s,x,y)/h(y)\ud{s}$ and then attempt \textit{to define} $\Psi_t(x)=\partial \mu_t(x)/\partial t$.  
\end{remark}

\end{document}